\documentclass[graybox]{svmult}

\usepackage{mathptmx}       
\usepackage{helvet}         
\usepackage{courier}        
\usepackage{makeidx}         
\usepackage{graphicx}        
\usepackage{multicol}        

\usepackage{amsmath}
\usepackage{amssymb}
\usepackage{amscd}
\usepackage{indentfirst}

\newtheorem{thm}{Theorem}[section]   
\newtheorem{prop}[thm]{Proposition}
\newtheorem{lem}{Lemma}

\newtheorem{Def}[thm]{Definition}

\newcommand{\oo} {{\omega}}

\newcommand{\bi} {{\beta}}
\newcommand{\ga} {{\gamma}}
\newcommand{\al} {{\alpha}}
\newcommand{\Ga} {{\varGamma}}

\newcommand{\ld} {{\ldots}}
\newcommand{\sm} {{\smallsetminus}}
\newcommand{\thi} {{\theta}}

\newcommand{\De} {{\varDelta}}

\newcommand{\la} {{\lambda}}
\newcommand{\el} {{\ell}}

\newcommand{\mi} {{\mu}}

\newcommand{\dis}{\displaystyle}

\newcommand{\ra}{{\rightarrow}}

\newcommand{\qb}{$\quad\blacksquare$}
\def\1{\it1\hspace*{-0.150cm}{\footnotesize{I}}}

\newcommand{\ca}{{\cal{A}}}

\newcommand{\cb}{{\cal{B}}}

\def\R{{\mathbb{R}}}
\def\C{{\mathbb{C}}}

\def\N{{\mathbb{N}}}

\def\N{{\Bbb N}}

\def\C{{\Bbb C}}
\def\1{{\bf 1}}

\makeindex             

\begin{document}

\title*{A Solution of polynomial equations}
\titlerunning{A Solution of polynomial equations}
\author{N. Tsirivas}
\authorrunning{N.~Tsirivas }
\institute{N.~Tsirivas \at Diovouniotou 30-32, T.K. 11741, Athens, Greece\\
\email{tsirivas66@yahoo.com} }
\maketitle
\abstract{We present a method for the solution of polynomial equations. We do not intend to present one more method among several others, because today there are many excellent methods. Our main aim is educational. Here we attempt to present a method with elementary tools in order to be understood and useful by students and educators. For this reason, we provide a self contained approach. Our method is a variation of the well known method of resultant, that has its origin back to Euler. Our goal, in the present paper, is in the spirit of calculus and secondary school mathematics. An extensive discussion of the theory of zeros of polynomials and extremal problems for polynomials the reader can find in the books \cite{10} and \cite{13}.}\vspace*{0.2cm}
\bigskip\\
\noindent
{\bf Introduction}\smallskip

It is well known that many problems in Physics, Chemistry and Science lead generally to a polynomial equation.

In pure mathematics also, there are classical problems that lead to a polynomial equation.

Let us give two examples:

1) If we are to compute the integral $\dis\int^\bi_\al\dfrac{p(x)}{q(x)}dx$, where $\al,\bi\in\R$, $\al<\bi$, and $p(x),q(x)$ are two real polynomials of one variable, and $q(x)$ is a non-zero polynomial that does not have any root in the interval $[\al,\bi]$, then we are led to the problem of finding the real roots of $q(x)$.\smallskip

2) Let $n\in\N$, $a_i\in\R$ for $i=1,\ld,n$, where $\N,\R$ are the sets of natural and real numbers respectively.

We can consider the differential equation
\[
a_ny^{(n)}+a_{n-1}y^{(n-1)}+\cdots+a_1y+a_0=0,
\]
where $y$ is the unknown function.

In order to solve this simple equation we have to find all the roots of the polynomial
\[
p(x)=a_nx^n+a_{n-1}x^{n-1}+\cdots+a_1x+a_0.
\]
So, the utility to solve a polynomial equation, or in other words to find the roots of a polynomial is undoubted. This problem is a very old, classical problem in mathematics and Numerical Analysis, especially. For this reason, there exist many methods that solve it.

However, if a scientist wants to solve an equation for his work, it is sufficient to use programs as ``mathematica'' and ``maple'', nowadays. So, the utility of the problem has an other direction, which is the finding of better algorithms and programs. This is the main line of research in the area experts, nowadays.

We are moving in an other direction in this paper.

Our main aim is mainly educational.

In this paper we present a method of solving a polynomial equation with full details for educational reasons so that a student of positive sciences can improve the level of knowledge in the subject. First of all, let us state our problem. We denote $\C$ as the set of complex numbers. Let $n\in\N$ and $a_i\in\C$ for $i=0,1,\ld,n$. We then consider the polynomial
\[
p(z)=a_nz^n+a_{n-1}z^{n-1}+\cdots+a_1z+a_0,
\]
that is a polynomial of one complex variable $z$ with complex coefficients. We suppose that $a_n\neq0$. The natural number $n$ is called the degree of $p(z)$ and it is denoted by $degp(z)=n$. The number $a\in\C$ is a root of $p(z)$ when it is applicable: $p(a)=0$. Our problem is to find all the roots of $p(z)$, or in other words to solve the equation, $p(z)=0$. Polynomials are simple and specific functions that have the following fundamental property:\medskip\\
\noindent
{\bf Fundamental Theorem of Algebra}. Every polynomial of one complex variable with complex coefficients and a degree greater or equal to one has at least one root in $\C$.

This result is central. It is the basis of our method.

However, even if this theorem is fundamental, its proof is not trivial. Its simplest proof comes from complex analysis that many students do not learn in university. In the appendix we give one of the simplest proofs of the fundamental Theorem.

Many of the best methods of our problem are iterative. They are based on the construction of specific sequences that approach to the roots of the supposed polynomial. Our method here uses algebra as much as possible, and when algebra cannot go further, analysis takes its role in solving the problem. Here we do not deal with the problem of speed of convergence. We use numerical analysis as little as we can. It is sufficient for us to use the simplest method in order to find a root in a specific real open interval, the bisection method.

 Most of the books on numerical analysis describe the bisection method with details. For example see (\cite{8}, \cite{11}).

There are some formulas that provide bounds of the roots of a polynomial. A. Cauchy had given such a bound, see \cite{8}. In the frame of our method we provide such a bound.

There are some results that give information about the number of positive or real roots, for example Descart's law of signs and Sturm's sequence \cite{11}. A basic problem is to find disjoint real intervals so that every one of them contains one root exactly. There are, also, many methods for this.

Let us describe now, roughly, the stages of our method.

1) In the first stage we find all the real roots of a polynomial. For this reason we are based on two results. First of all the bisection method and secondly by the following result:

If we have a polynomial $p$ of one real variable with real coefficients with a degree greater or equal to one for which we know the roots of $p'$, then using the bisection method we can find all the real roots of $p$.

The first stage is simple. It uses only elementary knowledge and it is also convenient for students of secondary school!

We think that it is very useful for students of secondary school to know a method that find all the real roots of an arbitrary real polynomial with their knowledge base.

2) In the second stage, we provide a method that gives all the real roots of a system of the form:
\[
\left\{\begin{array}{c}
                                   p(x,y)=0 \\ [2ex]
                                   q(x,y)=0
                                 \end{array}\right. \eqno{\mbox{(A)}}
\]
where $p(x,y),q(x,y)$ are polynomials of two real variables $x$ and $y$ with real coefficients. Our method here is a variation of the well known method of resultant (see \cite{6}, \cite{13}), that has its origin in Euler. With this method the solution of the above system $A$ is reduced to the first stage. As in the first stage, the second stage is also convenient for students of secondary school, (except for Theorem 3.17 in our prerequisites).

3) In the third stage we show that the solution of our problem is reduced to the second stage.

So, roughly speaking, our main aim in this paper is to present a method that is in the frame of the usual lessons of calculus in secondary school or in university and present it with all the necessary details in order for it to be understand by students.

As for the notation. Let $p(x,y)$ be a polynomial of two real variables $x$ and $y$, with real coefficients. We denote $deg_xp(x,y)$ the greatest degree of $p(x,y)$ with respect to $x$ and $deg_yp(x,y)$ the greatest degree of $p(x,y)$ with respect to $y$. If $deg_xp(x,y)\ge1$ and $deg_yp(x,y)\ge1$, we call the polynomial $p(x,y)$ a pure polynomial. If $p(z)$, $q(z)$ are two complex polynomials, we write $p(z)\equiv q(z)$, when they are equal by identity. We also write $p(z)\equiv0$, when $p(z)$ is equal to zero polynomial by identity. We write $p(z)\not\equiv q(z)$ when, polynomials $p(z)$, $q(z)$, are not equal by identity and $p(z)\not\equiv0$, when $p(z)$ is not the zero polynomial.

There are many methods and algorithms to the solution of polynomial equations. Some of them are very old like the methods of Horner, Graeffe and Bernoulli, whereas today there are some others like the methods of Rutishauser, Lehmer, Lin, Bairstow, Bareiss and many others. Another method, similar to Bernoulli method is the QD method. A classical and popular method today is that of Muller. It is a general method, not only for polynomials.

The interested reader can find the details of some of the above methods in the books of our references, see \cite{1}, \cite{3}, \cite{4}, \cite{7}, \cite{9}, \cite{10}, \cite{11} and \cite{12}. As we said there exist many algorithms and programs to our problem.

One of the best is the subroutine ZEROIN. One can find the details of this program in \cite{4}.

As we said formerly, the basis of our method is the resultant (or eliminent). With this method we can convert a system of polynomial equations in one equation with only one unknown!

Theoretically, we can succeed in that, but the complexity of calculations is enormous, so its value today is only for polynomial equations with a low degree, and is used as a theoretical tool. For details of the resultant see \cite{6}, \cite{12}. Apart from this there are some cases where the resultant fails. This can happen, for example, when we have to solve a system of two equations with two unknowns and one of the two equations is a multiple of the other, and the system has a finite number of solutions. See, for example, the equation: $(x^2-1)^2+(y^2-2)^2=0$, that has the set of solutions
\[
L=\{(1,\sqrt{2}), (1,-\sqrt{2}), (-1,\sqrt{2}), (-1,-\sqrt{2})\}.
\]
We describe with details how we handle these cases in our method here. An alternate method for our problem is to solve it with Gr\"{o}bner bases. Gr\"{o}bner bases is a method that was developed in 1960 for the division of polynomials with more than one variables. With Gr\"{o}bner bases we can also convert a system of equations in an equation with only one unknown, as the resultant does. This is the main application of Gr\"{o}bner bases. This can be done in most cases.

However, there are some cases where Gr\"{o}bner bases fail to succeed in the above, like the above case.

For Gr\"{o}bner bases, see \cite{2}. Many books of secondary school contain the elementary theory of polynomials and Euclidean division that we refer to in our prerequisites.

The structure of our paper is as follows:

In the first paragraph we give a roughly description of our method. In the second paragraph we give the complete description of our method. In the third paragraph we collect all the prerequisites tools of our method from Algebra and Analysis and we present them with all the necessary proofs, especially for results that someone cannot find easily in books.

Finally, in the last paragraph 4 (Appendix) we give one of the simplest proofs of the fundamental Theorem of Algebra that one cannot find easily in books.

We, also, give a short description of the solution of binomial equation: $x^n=a$, where $n\in\N$, $n\ge2$ and $a$ is a positive number.
\section{General description of the solution of our problem} \label{sec1}
\noindent

For methodological reasons, we divide the solution of our problem in three stages.
\subsection{First stage} In this stage we find all the real roots of the polynomial equation
\[
a_vx^v+a_{v-1}x^{v-1}+\cdots+a_1x+a_0=0,
\]
where $a_i\in\R$, for every $i=0,1,\ld,v$, where $v\in\N$.
\subsection{Second stage} Let $p_1(x,y)$, $p_2(x,y)$ be two polynomials of two real variables $x$ and $y$ whose coefficients are in $\R$. We consider the system of equations.
\[
\left\{\begin{array}{cc}
         p_1(x,y)=0 & (1) \\[2ex]
         p_2(x,y)=0 & (2)
       \end{array}\right. \eqno{\mbox{(A)}}
\]
Let $L_A$ be the set of solutions of the above system (A), in $\R^2$. That is, we consider the set
\[
L_A:=\{(x,y)\in\R^2|p_1(x,y)=0 \ \ \text{and} \ \ p_2(x,y)=0\},
\]
of solutions of the above system (A), in $\R^2$. In the second stage we find the set $L_A$ under the following supposition (S)\medskip\\
(S): {\bf Supposition}: We suppose that the set $L_A$ is finite.

That is, we solve the above system (A), in $\R^2$, in the case of supposition (S) holds. We note, that we succeed the second stage using first stage.
\subsection{Third stage} In the third stage we completely solve our initial problem of finding all the roots of the polynomial equation $a_nz^n+a_{n-1}z^{n-1}+\cdots+a_1z+a_0=0$, where $n\in\N$,\linebreak $a_i\in\C$, for every $i=0,1,\ld,n$, using the previous two stages.

The first stage is the analytical part, whereas the second and third stages are the algebraic parts of our method. The prerequisites of our method are few. Elementary calculus and the elementary linear algebra of secondary school are enough, except only for a specific case, where we use Theorem 3.17 from our prerequisites, (a very well-known result from calculus of several variables).

In the following paragraph, we give the complete description of our method.
\section{Complete description of our method}\label{sec2}
\subsection{First Stage} Let $a_i\in\R$, for $i=0,1,\ld,v$, $v\in\N$ and a polynomial
\[
p(x)=a_vx^v+a_{v-1}x^{v-1}+\cdots+a_1x+a_0,
\]
where $a_v\neq0$, so $deg p(x)=v$.

Here we find all the real roots of $p(x)$. If $v=1$, or $v=2$ we know how to find the real roots of $p(x)$ from secondary school. Let us suppose that $v\ge 3$. We find all the real roots of $p'$ (if any) and then we find the roots of $p$ by applying basic Lemma 3.8 or Corollaries 3.9, 3.10.

More generally, we suppose that $p$ has degree $v\in\N$, $v\ge3$. We consider polynomials $p',p'',\ld,p^{(v-3)}$ $p^{(v-2)}$. Polynomial $p'$ has degree $v-1$, $p''$ has degree $v-2$, polynomial $p^{(v-2)}$ has degree 2.

We find the roots of $p^{(v-2)}$ (if any). After using basic Lemma 3.8, or Corollaries 3.9, 3.10 we find the roots of $p^{(v-3)}$ and going inductively, after a finite number of steps, we find the roots of $p'$ and finally in the same way the roots of $p$, and thus we complete our first stage.
\subsection{Second stage} We will now consider the system of two polynomials $p_1(x,y),p_2(x,y)$ of two real variables $x$ and $y$ with coefficients in $\R$. We solve the system (A) where
\[
\left\{\begin{array}{cc}
         p_1(x,y)=0 & (1) \\[2ex]
         p_2(x,y)=0 & (2)
       \end{array}\right. \eqno{\mbox{(A)}}
\]
We solve system (A) with the following \medskip\\
{\bf Supposition}: We suppose that system (A) has a finite number of solutions, that is, the set
\[
L_A=\{(x,y)\in\R^2|p_1(x,y)=p_2(x,y)=0\}
\]
is non-void and finite.

Firstly, we notice that one of the polynomials $p_1(x,y)$, $p_2(x,y)$, at least, is non zero, or else if $p_1(x,y)\equiv p_2(x,y)\equiv0$ for every $(x,y)\in\R$, then we have $L_A=\R^2$, that is false because the set $L_A$ is finite. We will examine some cases.

First of all, we suppose that at least one of the polynomials is of one variable only. We can distinguish some cases here. Let $p_1(x,y)\equiv q_1(x)$, $p_2(x,y)\equiv q_2(x)$. Then, we solve the equations $q_1(x)=0$ and $q_2(x)=0$ with the method of the first stage and after we conclude that set $L_A$ is the set of all $(x,y)$, where $x$ is one of the common solutions of equations $q_1(x)=0$ and $q_2(x)=0$ and $y\in\R$, that is $L_A$ is an infinite set, which is false by our supposition. So this case itself cannot occur.
Similarly, we can't have the case where $p_1(x,y)\equiv r_1(y)$ and $p_2(x,y)\equiv r_2(y)$. Now we consider the case where:
\[
p_1(x,y)\equiv q_1(x) \ \ \text{and} \ \ p_2(x,y)\equiv q_2(y).
\]
Then, we can solve the equations $q_1(x)=0$ and $q_2(y)=0$ with the method of the first stage, and we find the finite sets $A_1=\{\rho_1,\rho_2,\ld,\rho_v\}$ and $B_1=\{\la_1,\la_2,\ld,\la_m\}$, $A_1\cup B_1\subseteq\R$ where $A_1$ is the set of roots of $q_1$ and $B_1$ is the set of roots of $q_2$, $v,m\in\N$. Then, we have $L_A=\{(\rho_i,\la_j), i=1,\ld,v, j=,\ld,m\}$. In a similar way, we can solve the system A, when $p_1(x,y)=r_1(y)$ and $p_2(x,y)=r_2(x)$, for some polynomials $r_1(y), r_2(x)$.

Now, we consider the case where $p_1(x,y),p_2(x,y)$ is two pure polynomials.\medskip

(i) The simplest case is when $deg_yp_1(x,y)=deg_yp_2(x,y)=1$. Then we have:
\[
\begin{array}{lc}
  p_1(x,y)=\al_1(x)y+\al_2(x) & \text{and} \\ [2ex]
  p_2(x,y)=\bi_1(x)y+\bi_2(x), &
\end{array}
\]
where $\al_1(x),\al_2(x),\bi_1(x),\bi_2(x)$ are some polynomials of real variable $x$ only and $\al_1(x)\not\equiv0$ an $\bi_1(x)\not\equiv0$, because $p_1(x,y)$, $p_2(x,y)$ are pure polynomials.

So we have to solve the system:
\[
\left\{\begin{array}{cc}
         \al_1(x)y+\al_2(x)=0 & (3) \\[2ex]
         \bi_1(x)y+\bi_2(x)=0 & (4)
       \end{array}\right.
\]
We can distinguish some cases here. There exists a $(x_0,y_0)\in L_A$, so that:
\begin{enumerate}
\item[1)] $\al_1(x_0)=\bi_1(x_0)=0$. Then with (3) and (4), we get: $\al_2(x_0)=\bi_2(x_0)=0$. We get $(x_0,y)\in L_A$ for every $y\in\R$, that is false because $L_A$ is finite. So, this case can not occur.
\item[2)] $\al_1(x_0)=0$ and $\bi_1(x_0)\neq0$. Then with (4) we take:

    $y_0=-\dfrac{\bi_2(x_0)}{\bi_1(x_0)}$ (5). With (3) we have: $\al_2(x_0)=0$.

    So, in this case we find the common roots of polynomials $\al_1(x)$ and $\al_2(x)$, and for every common root $x_0$ of $\al_1(x)$ and $\al_2(x)$, so that $\bi_1(x_0)\neq0$, the couple $(x_0,y_0)\in L_A$, where $y_0$ is given from (5). Of course we find the real roots of polynomials $\al_1(x)$ and $\al_2(x)$ with the method of our first stage. In a similar way we find the roots $(x_0,y_0)\in L_A$, so that, $\al_1(x_0)\neq0$ and $\bi_1(x_0)=0$.
\item[3)] $\al_1(x_0)\neq0$ and $\bi_1(x_0)\neq0$.

Here, we have some cases:
\begin{itemize}
\item[(i)] $\al_2(x_0)=\bi_2(x_0)=0$. Then with (3), (4) and our supposition, we get: $y=0$. So, in this case we find the common roots of polynomials $\al_2(x)$ and $\bi_2(x)$, so that they are not roots of polynomials $\al_1(x)$ and $\bi_1(x)$, with the method of the first stage. If $x_0$ is such a root, that is: $\al_2(x_0)=\bi_2(x_0)=0$ and $\al_1(x_0)\neq0$ and $\bi_2(x_0)\neq0$, then: $(x_0,0)\in L_A$.
\item[(ii)]\ \  $\al_2(x_0)=0$ and $\bi_2(x_0)\neq0$.

Then, with (3), and the facts $\al_2(x_0)=0$ and $\al_1(x_0)\neq0$, we get: $y=0$. Then, because $y=0$, by (4) we get $\bi_2(x_0)=0$, that is a contradiction by our supposition. So, this case cannot occur.
\item[(iii)]\ \  $\al_2(x_0)\neq0$ and $\bi_2(x_0)=0$. As in the previous case (ii), this case cannot occur.
\item[(iv)] \ \ $\al_2(x_0)\neq0$ and $\bi_2(x_0)\neq0$.

Then, with (3) and (4) we get:\smallskip

$y_0=-\dfrac{\al_2(x_0)}{\al_1(x_0)}$ (6) and $y_0=-\dfrac{\bi_2(x_0)}{\bi_1(x_0)}$ (7).\smallskip

With (6) and (7) we get:\smallskip

$-\dfrac{\al_2(x_0)}{\al_1(x_0)}=-\dfrac{\bi_2(x_0)}{\bi_1(x_0)}\Leftrightarrow
\al_2(x_0)\bi_1(x_0)-\al_1(x_0)\bi_2(x_0)=0$ (8)
\end{itemize}
\end{enumerate}

Now, we can consider two systems of relations (A1) and (A2) as follows:
\[
\left\{\begin{array}{l}
         \al_1(x)y+\al_2(x)=0, \quad (i) \\ [1.5ex]
         \bi_1(x)y+\bi_2(x)=0,\quad (ii) \\ [1.5ex]
         \al_1(x)\neq0, \al_2(x)\neq0, \bi_1(x)\neq0, \bi_2(x)\neq0
       \end{array}\right.  \eqno{\mbox{(A$_1$)}}
\]
and
\[
\left\{\begin{array}{l}
         \al_2(x)\bi_1(x)-\al_1(x)\bi_2(x)=0, \quad (i) \\ [1.5ex]
         y=-\dfrac{\al_2(x)}{\al_1(x)},\quad (ii) \\ [1.5ex]
         \al_1(x)\neq0, \al_2(x)\neq0, \bi_1(x)\neq0, \bi_2(x)\neq0
       \end{array}\right.  \eqno{\mbox{(A$_2$)}}
\]
Let $L_{A_1}$, $L_{A_2}$ be the two set of solutions of systems $A_1$ and $A_2$ respectively. We prove that $L_{A_1}=L_{A_2}$.

By previous procedure and equalities (6) and (8) we get:
\[
L_{A_1}\subseteq L_{A_2} \quad (9)
\]
Now let $(x,y)\in L_{A_2}$. Then equality ii) of $A_2$ gives equality ii) of $A_1$. By equality (i) of $(A_2)$ we get: $\al_2(x)\bi_1(x)=\al_1(x)\bi_2(x)$ and by the fact that $\al_1(x)\neq0$ and $\bi_1(x)\neq0$, we get:
\[
-\frac{\al_2(x)}{\al_1(x)}=-\frac{\bi_2(x)}{\bi_1(x)}. \quad (10)
\]
Through the equality (ii) of $(A_2)$ and (10) we get:
\[
y=-\frac{\bi_2(x)}{\bi_1(x)}. \quad (11)
\]
Equality (11) gives equality (ii) of $(A_1)$. So we have $(x,y)\in L_{A_1}$, that is $L_{A_2}\subseteq L_{A_1}$ \ \ (12).

By (9) and (12) we get: $L_{A_1}=L_{A_2}$.

So, we proved that in order to solve system $(A_1)$ it suffices to solve system $(A_2)$. Thus, we solve equation (i) of $(A_2)$ with the method of the first stage, and for every root $x$ of polynomial $\al_2(x)\bi_1(x)-\al_1(x)\bi_2(x)$ so that $\al_1(x)\neq0$, $\al_2(x)\neq0$,\linebreak $\bi_1(x)\neq0$, $\bi_2(x)\neq0$, we get the respective $y$ from equality (ii) of $(A_2)$.

So far we have completely solved the system (A), in the case of\\ $deg_yp_1(x,y)=deg_yp_2(x,y)=1$.

For the sequel we solve the case ii) where $deg_yp_1(x,y)\le2$ and $deg_yp_2(x,y)\le2$ and $p_1(x,y),p_2(x,y)$ are two pure polynomials. Of course, we have $deg_yp_1(x,y)\ge1$ and $deg_yp_2(x,y)\ge1$, because $p_1(x,y),p_2(x,y)$ are pure polynomials.\medskip

We have already examined the case $deg_yp_1(x,y)=deg_yp_2(x,y)=1$.

So we, here, examine the case where at least one of natural numbers $deg_yp(x,y)$, $deg_yp_2(x,y)$ are equal to 2.

We examine, firstly, the case where:
\[
deg_yp_1(x,y)=2 \ \ \text{and} \ \ deg_yp_1(x,y)=1.
\]
Then, we can write the system (A) as follows:
\[
\left\{\begin{array}{rc}
         \al_2(x)y^2+\al_1(x)y+\al_0(x)=0 & (13) \\ [2ex]
         \bi_1(x)y+\bi_0(x)=0 & (14) \\
       \end{array}\right. \eqno{\mbox{(A)}}
\]
If $\al_2(x)\equiv0$, we have the previous system. So we suppose that $\al_2(x)\not\equiv0$.

Now let some $(x_0,y_0)\in L_A$ as above. We distinguish some cases:
\begin{enumerate}
\item[1)] $\al_2(x_0)=0$. Then, we solve the system $\left\{\begin{array}{l}
\al_1(x)y+\al_0(x)=0 \\
\bi_1(x)y+\bi_0(x)=0 \end{array}\right.$ as previously and we take only the solutions $(x,y)$ of this system so that $\al_2(x)=0$ holds, solving the equation $\al_2(x)=0$ with the method of the first stage.
\item[2)] $\al_2(x_0)\neq0$. We distinguish some cases.
\begin{itemize}
\item[(i)] \ \ $\al_1(x_0)=\bi_1(x_0)=0$. Then we have to solve the system
\[
\left\{\begin{array}{rc}
  \al_2(x)y^2+\al_2(x)=0 & (15) \\ [2ex]
  \bi_0(x)=0 & (16)
\end{array}\right.. \eqno{\mbox{(B$_1$)}}
\]
By (15) we take:
\[
y^2=-\frac{\al_0(x)}{\al_2(x)}. \quad (17)
\]
So, in order to solve this system we do the following:

First of all we find all the common roots $x$ of three polynomials $\al_1(x)$, $\bi_1(x)$, $\bi_0(x)$ that are not roots of polynomial $\al_2(x)$.

If $x\in\R$ and $\al_1(x)=\bi_1(x)=\bi_0(x)=0$ and $\al_2(x)\neq0$, we consider the number $-\dfrac{\al_0(x)}{\al_2(x)}$. If $-\dfrac{\al_0(x)}{\al_2(x)}\ge0$, then we set

$\left(y_1=\sqrt{-\dfrac{\al_0(x)}{\al_2(x)}}\right.$ and $y_2=-\sqrt{-\dfrac{\al_0(x)}{\al_2(x)}}$, if $\left.-\dfrac{\al_0(x)}{\al_2(x)}>0\right)$ and\\ $(y=0$ if $\al_0(x)=0)$, and then under the above conditions $(x,y)\in L_A$.

We find the roots of polynomials $\al_1(x)$, $\bi_1(x)$, $\bi_0(x)$ with the method of the first stage.

Of course, if we cannot find couples $(x,y)\in\R^2$ so that all the above conditions hold, this means, that we do not have solutions to this case.
\item[(ii)] \ \ $\al_1(x_0)=0$, $\bi_1(x_0)\neq0$.

We consider the system:
\[
\left\{\begin{array}{rc}

         \al_2(x)y^2+\al_0(x)=0 & (17) \\ [2ex]
         \bi_1(x)y+\bi_0(x)=0 & (18)
       \end{array}\right.. \eqno{\mbox{(B$_1$)}}
\]
Through (17) and (18) we get:
\[
y^2=-\frac{\al_2(x)}{\al_2(x)} \quad (19)
\]
\[
y=-\frac{\bi_0(x)}{\bi_1(x)} \quad (20) \ \ \Rightarrow \ \ y^2=\bigg(\frac{\bi_0(x)}{\bi_1(x)}\bigg)^2 \quad (21)
\]
Through (19) and (21) we get:
\[
-\frac{\al_0(x)}{\al_2(x)}=\bigg(\frac{\bi_0(x)}{\bi_1(x)}\bigg)^2\Leftrightarrow\al_2(x)
\bi_0(x)^2+\al_0(x)\bi_1(x)^2=0. \quad (22)
\]
 From the above, in order to find a solution of system $(B_2)$ we do the following:

We find all the common roots of two polynomials $\al_2(x)\bi_0(x)^2+\al_0(x)\bi_1(x)^2$ and $\al_1(x)$, that are not roots of polynomials $\al_2(x)$ and $\bi_1(x)$ (if any). Let $x$ be such a root. We set $y=-\dfrac{\bi_0(x)}{\bi_1(x)}$, and then $(x,y)$ is a solution of $(B_2)$ and we get all the other solutions of $(B_2)$ in the same way.
\item[(iii)] \ \ $\al_1(x_0)\neq0$, $\bi_1(x_0)=0$.

Then, through (14) we get $\bi_0(x_0)=0$. So, in order to solve system (A) in this case, we do the following.

We find all the common roots (if any)  $x$ of polynomials $\bi_1(x)$, $\bi_0(x)$, so that $\al_2(x)\neq0$ and $\al_1(x)\neq0$. Of course this is a finite set of numbers $x$.

For such a root $x_0$ we solve the equation $\
\al_2(x_0)y^2+\al_1(x_0)y+\al_0(x)=0$ and we find the respective number $y$ (if any).

All these couples $(x,y)\in\R^2$ (if any) are the set of solutions of system (A) in this case.
\item[(iv)] \ \ $\al_1(x_0)\neq0$, $\bi_1(x_0)\neq0$.

We leave this case for later. In a similar way we examine the case where $deg_yp_1(x,y)=1$ and $deg_yp_2(x,y)=2$.
\end{itemize}
\end{enumerate}

Now, we examine the case where:
\[
deg_yp_1(x,y)=deg_yp_2(x,y)=2.
\]
We have the system:
\[
\left\{\begin{array}{cc}
         \al_2(x)y^2+\al_1(x)y+\al_0(x)=0 & (23) \\ [2ex]
         \bi_2(x)y^2+\bi_1(x)y+\bi_0(x)=0 & (24)
       \end{array}\right.\eqno{\mbox{(B$_3$)}}
\]
Here we examine some cases:
\begin{enumerate}
\item[1)] Let $(x_0,y_0)\in L_{B_3}$.

If $\al_2(x_0)=0$, or $\bi_2(x_0)=0$ we have a system as in the previous case.

So, we suppose that:
\[
\al_2(x_0)\neq0 \ \  \text{and} \ \ \bi_2(x_0)\neq0.
\]
Now, we can distinguish some cases:
\begin{itemize}
\item[i)] \ \ $\al_1(x_0)=\bi_1(x_0)=0$.

So, we are to solve the system:
\[
\begin{array}{rl}
  \al_2(x_0)y^2+\al_0(x_0)=0 & (25) \ \ \text{and} \\[1.5ex]
  \bi_2(x_0)y^2+\bi_0(x_0)=0 & (26).
\end{array}
\]
Through (25) we have $y^2=-\dfrac{\al_0(x_0)}{\al_2(x_0)}$ (27) and by (26) we get
\[
y^2=-\frac{\bi_0(x_0)}{\bi_2(x_0)}. \eqno{\mbox{(28)}}
\]
Through (27) and (28) we get:
\[
-\frac{\al_0(x_0)}{\al_2(x_0)}=-\frac{\bi_0(x_0)}{\bi_2(x_0)}\Leftrightarrow
\al_0(x_0)\bi_2(x_0)-\bi_0(x_0)\al_2(x_0)=0.    \eqno{\mbox{(29)}}
\]
From the above, we have the following solution:

We find the common roots of polynomials\\ $\al_1(x),\bi_1(x),\al_0(x)\bi_2(x)-\bi_0(x)\al_2(x)$ so that $\al_2(x)\neq0$ and $\bi_2(x)\neq0$.

We suppose that there exists such a root $x_0$. If $\al_0(x_0)=0$, we get $y_0=0$.\\
If $\dfrac{\al_0(x_0)}{\al_2(x_0)}<0$, we consider
\[
y_1=\sqrt{-\frac{\al_0(x_0)}{\al_2(x_0)}}, \ \ y_2=-\sqrt{-\frac{\al_0(x_0)}{\al_2(x_0)}},
\]
and $(x_0,y_1),(x_0,y_2)\in L_{B_3}$. We get all the other solutions of $B_3$ in the same way. Of course, if $(x,y)$ does not exist with the above conditions, we do not have solutions of $B_3$ in this case.
\item[ii)] \ \ $\al_1(x_0)=0$ and $\bi_1(x_0)\neq0$.

We will postpone this case for later.
\item[iii)] \ \ $\al_1(x_0)\neq0$ and $\bi_1(x_0)=0$.

We will also postpone this case for later.

iv) \ \ $\al_1(x_0)\neq0$ and $\bi_1(x_0)\neq0$.

This is the central case of system $(B_3)$.
\end{itemize}
\end{enumerate}

We consider the number:
\[
D=\al_2(x_0)\bi_1(x_0)-\al_1(x_0)\bi_2(x_0)=\left|\begin{array}{ccc}
                                                   \al_2(x_0) & & \al_1(x_0) \\ [2ex]
                                                    \bi_2(x_0) & & \bi_1(x_0)
                                                  \end{array}\right|
\]
that we call it: the determinant of system $(B_3)$.

We distinguish two cases:\medskip\\
a) $ D\neq0$.\\
We consider the linear system:
\[
\left\{\begin{array}{cl}
         \al_2(x_0)z+\al_1(x_0)\oo=-\al_0(x_0) & (30) \\ [2ex]
          \bi_2(x_0)z+\bi_1(x_0)\oo=-\bi_0(x_0)& (31)
       \end{array}\right.\eqno{\mbox{(B$_4$)}}
\]
This linear system has determinant $D\neq0$, so, it has exactly one solution.

We set:
\[
D_1=\left|\begin{array}{ccc}
            -\al_0(x_0) &&\al_1(x_0)  \\ [2ex]
            -\bi_0(x_0) && \bi_1(x_0)
                      \end{array} \right|
=\al_1(x_0)\bi_0(x_0)-\al_0(x_0)\bi_1(x_0) \eqno{\mbox{(32)}}
\]
and
\[
D_2=\left|\begin{array}{ccc}
            \al_2(x_0) &&-\al_0(x_0)  \\ [2ex]
            \bi_2(x_0) && -\bi_0(x_0)
                      \end{array} \right|
=\al_0(x_0)\bi_2(x_0)-\al_2(x_0)\bi_2(x_0) \eqno{\mbox{(33)}}
\]
Through Cramer's law of linear algebra we get the unique solution $(z,\oo)$ of system $B_4$, that is
\[
z=\frac{D_1}{D} \ \ \text{and} \ \ \oo=\frac{D_2}{D}.
\]
From our supposition the couple $(x_0,y_0)\in L_{B_3}$. This means that the numbers $y^2_0$ and $y_0$ satisfy equations (23) and (24) of $(B_3)$, or differently, in other words the couple $(y^2_0,y_0)$ is a solution of the linear system $(B_4)$. But because of our supposition $D\neq0$, the couple $(z,\oo)$, where $z=\dfrac{D_1}{D}$ \ \ (34) and $\oo=\dfrac{D_2}{D}$ \ \ (35), is the unique solution of system $(B_4)$, as it is well known in  linear algebra. So, we have $z=y^2_0$ and $\oo=y_0$, and by (34) and (35) we get:
\[
y^2_0=\frac{D_1}{D} \quad (36) \ \ \text{and} \ \ y_0=\frac{D_2}{D} \quad (37)
\]
Now, we exploit the inner relation that numbers $y_0$ and $y^2_0$ have, that is:
\[
y^2_0=y_0\cdot y_0.  \eqno{\mbox{(38)}}
\]
Replacing (36) and (37) in relation (38), we get:
\[
\frac{D_1}{D}=y_0\cdot\frac{D_2}{D}\Rightarrow D_1-y_0D_2=0. \eqno{\mbox{(39)}}
\]
By (37) we have
\[
Dy_0-D_2=0\eqno{\mbox{(40)}}
\]
So, the couple $(x_0,y_0)\in L_{B_3}$ also satisfies the system:
\[
\left\{\begin{array}{rcl}
         D_1-yD_2=0 && (39) \\  [2ex]
         Dy-D_2=0 && (40)
       \end{array}\right.
\]
From the above, we have the two systems:
\[
\left\{\begin{array}{l}
         \al_2(x)y^2+\al_1(x)y+\al_0(x)=0 \\ [1.5ex]
         \bi_2(x)y^2+\bi_1(x)y+\bi_0(x)=0  \\ [1.5ex]
         \al_2(x)\neq0,\bi_2(x)\neq0,\;\al_1(x)\neq,\;\bi_1(x)\neq0,  \\ [1.5ex]
         D=\left|\begin{array}{cc}
                   \al_2(x) & \al_1(x) \\ [2ex]
                   \bi_2((x) & \bi_1(x)
                 \end{array}\right|\neq0
         \end{array}\right. \eqno{\mbox{(B$_5$)}}
\]
and
\[
\left\{\begin{array}{l}
         D_1=yD_2=0 \\ [1.5ex]
         Dy-D_2=0 \\ [1.5ex]
         \al_2(x)\neq0,\;\bi_2(x)\neq0,\;\al_1(x)\neq0,\;\bi_1(x)\neq0, \\ [1.5ex]
         D\neq0
       \end{array}\right..  \eqno{\mbox{(B$_6$)}}
\]
Let $L_{B_5}$, $L_{B_6}$ be the set of solutions of systems $(B_5)$ and $(B_6)$. We now show that $L_{B_5}=L_{B_6}$. Of course we have $L_{B_5}\subseteq L_{B_6}$ from the previous procedure, because we obtained equalities (39) and (40) of system $(B_6)$ from system $B_5$.

Reversely, let $(x_0,y_0)\in L_{B_6}$. From the first two equalities of $(B_6)$ we get:
\[
y_0=\frac{D_1}{D_2} \ \ \text{and} \ \ y_0=\frac{D_2}{D} \eqno{\mbox{(37)}}
\]
We multiply these equalities and we take: $y^2_0=\dfrac{D_1}{D}$ (36).

Now, we consider the linear system $(B_4)$. Because $D\neq0$ (by our supposition), this system has a unique solution  $(z,\oo)=\Big(\dfrac{D_1}{D},\dfrac{D_2}{D}\Big)$, (41) as it is well known, in Cramer's law.

From (36), (37) and (41) we have: $z=y^2_0$ (42) and $\oo=y_0$ (43).

Replacing (42) and (43) in (30) and (31) of $(B_4)$ we get the first two equalities of $(B_5)$, that is $(x_0,y_0)\in L_{B_5}$. So, we have: $L_{B_6}\subseteq L_{B_5}$.\\
From the above we have $L_{B_5}=L_{B_6}$. So, we are led to solve system $B_6$, which we have examined previously, in the system: $\left.\begin{array}{l}
                                                   p_1(x,y)=0 \\[1.5ex]
                                                   p_2(x,y)=0
                                                 \end{array}\right\}$, where
$\begin{array}{l}
   deg_yp_1(x,y)\le1 \ \ \text{and} \\ [1.5ex]
   deg_yp_2(x,y)\le1.
 \end{array}$ \medskip\\
b) $D=0$\\
The solution of  these cases is taken as follows:

We take the roots of polynomial $D$ $x_1$, that is $\al_2(x_1)\bi_1(x_1)-\al_1(x_1)\bi_2(x_1)=0$ so that $\al_2(x_1)\neq0$, $\bi_2(x_1)\neq0$, $\al_1(x_1)\neq0$ and $\bi_1(x_1)\neq0$. We get $y$ that satisfies one of the equations (23), or (24) of $(B_3)$, that is:
\[
\al_2(x_1)y^2+\al_1(x_1)y+\al_0(x_1)=0.
\]
This holds because the two equations (23) and (24) are equivalent, (as we have shown in prerequisites of linear algebra), and each of them is a multiple of the other.\vspace*{0.2cm}\\
\noindent
{\bf Remark 2.2.1.}
{\em We note that the three remaining cases we have left are similar to  case a) above where $D\neq0$}.\smallskip

So far, we have examined system (A), where $deg_yp_1(x,y)\le2$ and\linebreak $deg_yp_2(x,y)\le2$. We set
\[
m_0:=\max\{deg_yp_1(x,y),deg_yp_2(x,y)\}.
\]
We solve system (A) in the general case with induction above the number $m_0$. We have examined the cases where $m_0=1$ or $m_0=2$.

We suppose that for $k_0\in\N$, $k_0\ge3$, we have solved system (A) for every system so that $m_0\le k_1-1$. We now solve system (A) when $m_0=k_0$.

We can write polynomials $p_1(x,y),p_2(x,y)$ as follows:
\begin{align*}
\al_{m_0}(x)y^{m_0}+\al_{v_0}(x)y^{v_0}+q_1(x,y)&=p_1(x,y) \ \ \text{and} \ \ \bi_{n_0}(x)y^{n_0}+\bi_{\mi_0}(x)y^{\mi_0}+q_2(x,y)\\
&=p_2(x,y),
\end{align*}
where $v_0<m_0$, $v_0,m_0\in\N$, $deg_yq_1(x,y)<v_0$ and $\mi_0<n_0$, $\mi_0,n_0\in\N$, $n_0\le m_0$, $deg_yq_2(x,y)<\mi_0$.

So, the initial system can be written as follows:
\[
\left\{\begin{array}{rl}
       \al_{m_0}(x)y^{m_0}+\al_{v_0}(x)y^{v_0}+q_1(x,y)=0 & (1)   \\ [2ex]
       \bi_{n_0}(x)y^{n_0}+\bi_{\mi_0}(x)y^{\mi_0}+q_2(x,y)=0 & (2)
       \end{array}\right.    \eqno{\mbox{(A)}}
\]
where $\al_{m_0}(x),q_{v_0}(x),\bi_{n_0}(x),\bi_{\mi_0}(x)$ are polynomials of the real variable $x$ only and $q_1(x,y),q_2(x,y)$ be polynomials of real variables $x$ and $y$.

We, also, suppose that $a_{m_0}(x)\not\equiv0$. We can distinguish some cases as previously:\medskip

1) Let $\al_{v_0}(x)\equiv q_1(x,y)\equiv\bi_{n_0}(x)\equiv\bi_{\mi_0}(x)\equiv q_2(x,y)\equiv0$. Then we have the system: $\al_{m_0}(x)y^{m_0}=0$. If $m_0=0$, and $\al_{m_0}(x)=c\in\R$ then every couple $(x,0)\in L_A$ and the system has an infinite set of solutions, which is false. So, this case cannot occur.\medskip

2) $\bi_{n_0}(x)\equiv\bi_{\mi_0}(x)\equiv q_2(x,y)\equiv0$.

We will examine this case later.\medskip

3) $\bi_{n_0}(x)\not\equiv0$, $\al_{v_0}(x)\equiv q_1(x,y)\equiv\bi_{\mi_0}(x)\equiv q_2(x,y)\equiv0$.

We have the system
\[
\left\{\begin{array}{c}
         \al_{m_0}(x)y^{m_0}=0 \\ [2ex]
         \bi_{n_0}(x)y^{n_0}=0
       \end{array}\right.   \eqno{\mbox{(A)}}
\]
If $deg\al_{m_0}(x)=deg\bi_{n_0}(x)=0$, then any couple $(x,0)\in L_A$ and the set of solutions is infinite, which is false. So, this case cannot occur. \medskip

4) $\bi_{n_0}(x)\not\equiv0$, $\al_{v_0}(x)\equiv\bi_{\mi_0}(x)\equiv0\equiv q_1(x,y)$ and $q_2(x,y)\equiv r(x)\not\equiv0$.

So, we have the system:
\[
\left\{\begin{array}{rl}
  \al_{m_0}(x)y^{m_0}=0 & (3) \\ [2ex]
  \bi_{n_0}(x)y^{n_0}+r(x)=0 & (4)
\end{array}\right. \eqno{\mbox{(A)}}
\]
We can distinguish some cases here.

Firstly, we suppose that (A) has a solution $(x_0,y_0)\in L_A$.
\begin{enumerate}
\item[i)] $\al_{m_0}(x_0)=0=\bi_{n_0}(x_0)$. Then, of course $r(x_0)=0$. So, if the polynomials $\al_{m_0}(x), \bi_{n_0}(x),r(x)$, has a common root $x_0$, then any couple $(x_0,y)\in L_A$, for every $y\in\R$, which is false of course.

So, this case cannot happen.
\item[ii)] $\al_{m_0}(x_0)=0$, $\bi_{n_0}(x_0)\neq0$.

By (4), we take $y_0=-\dfrac{r(x_0)}{\bi_{n_0}(x_0)}$.

Thus, in this case we solve the system as follows:

We find the roots $x$ of $\al_{m_0}(x)$, so that $\bi_{n_0}(x)\neq0$. For every such root the couple $(x,y)=\Big(x,-\dfrac{r(x)}{\bi_{n_0}(x)}\Big)\in L_A$.

We get all the other solutions of this system in the same way.
\item[iii)] $\al_{m_0}(x_0)\neq0$ and $\bi_{n_0}(x_0)=0$.

Through (3), we get $y=0$. By (4), we take $r(x_0)=0$.

So, in this case we can solve the system as follows: We find the roots of $r(x)$ such that $\al_{m_0}(x)\neq0$ and $\bi_{n_0}(x)=0$. For every such root $x$, the couple $(x,0)\in L_A$.
\item[iv)] $\al_{m_0}(x_0)\neq0$ and $\bi_{n_0}(x_0)\neq0$.

By (3) we get $y_0=0$, and by (4), for $y_0=0$ we get $r(x_0)=0$.

So, in this case we can solve the system as follows: We find the roots $x$ of $r(x)$ so that $\al_{m_0}(x)\neq0$ and $\bi_{n_0}\neq0$. Then the couple $(x,0)$ is a solution of (A).
\item[v)] In a similar way we can solve a system of the form:
\[
\left\{\begin{array}{r}
         \al_{m_0}(x)y^{m_0}+r(x)=0 \\ [2ex]
         \bi_{n_0}(x)y^{n_0}=0.
       \end{array}\right.
\]
\end{enumerate}

5) $\bi_{n_0}(x)\not\equiv0$, $\al_{v_0}(x)\equiv\bi_{\mi_0}(x)\equiv0$, $q_1(x,y)\equiv r_v(x)\not\equiv0$, $q_2(x,y)\equiv r_2(x)\not\equiv0$.

So, we have the system:
\[
\left\{\begin{array}{c}
         \al_{m_0}(x)y^{m_0}+r_1(x)=0 \\ [2ex]
         \bi_{n_0}(x)y^{n_0}+r_2(x)=0
       \end{array}\right.    \eqno{\mbox{(A)}}
\]
Let $(x_0,y_0)\in L_A$.

If $r_1(x_0)=0$ or $r_2(x_0)=0$, then we have the system of the previous case 4. So, we suppose that: $r_1(x_0)\neq0$ and $r_2(x_0)\neq0$. We can distinguish some cases:
\begin{enumerate}
\item[(i)] $\al_{m_0}(x_0)=\bi_{n_0}(x_0)=0$

So, we have the system:
\[
\left\{\begin{array}{c}
         r_1(x)=0 \\ [2ex]
         r_2(x)=0
       \end{array}\right..  \eqno{\mbox{(A)}}
\]
Let $x_0$ be a common root of $r_1(x),r_2(x)$. Then, we have $(x_0,y)\in L_A$, for every $y\in\R$, which is false of course, because the set $L_A$ is finite.

So, this case cannot occur.
\item[(ii)] $\al_{m_0}(x_0)=0$ and $\bi_{n_0}(x_0)\neq0$. Then, we get $r_1(x_0)=0$ and $y^{n_0}_0=-\dfrac{r_2(x_0)}{\bi_{n_0}(x_0)}$, and if $n_0$ is odd we have $y_0=\sqrt[n_0]{-\dfrac{r_2(x_0)}{\bi_{n_0}(x_0)}}$ if $\dfrac{r_2(x_0)}{\bi_{n_0}(x_0)}\le0$ and $y_0=-\sqrt[n_0]{\dfrac{r_2(x_0)}{\bi_{n_0}(x_0)}}$\\ if $\dfrac{r_2(x_0)}{\bi_{n_0}(x_0)}>0$.

    If $n_0$ is even and $\dfrac{r_2(x_0)}{\bi_{n_0}(x_0)}\le0$, we have:
\[
y_1=\sqrt[n_0]{-\dfrac{r_2(x_0)}{\bi_{n_0}(x_0)}} \ \ \text{and} \ \ y_2=-\sqrt[n_0]{-\dfrac{r_2(x_0)}{\bi_{n_0}(x_0)}}.
\]
So, in this case we solve the system as follows:

We find the common roots of polynomials $\al_{m_0}(x)$ and $r_1(x_0)$, such that $\bi_{n_0}(x)\neq0$. Let $x_0$ such a root.

Then, if $n_0$ is odd and $\dfrac{r_2(x_0)}{\bi_{n_0}(x_0)}\le0$, then the couple $(x_0,y_0)\in L_A$, where
\[
y_0=\sqrt[n_0]{-\dfrac{r_2(x_0)}{\bi_{n_0}(x_0)}},
\]
whereas if $\dfrac{r_2(x_0)}{\bi_{n_0}(x_0)}>0$, then the couple $(x_0,y_0)\in L_A$, where
\[
y_0=-\sqrt[n_0]{\dfrac{r_2(x_0)}{\bi_{n_0}(x_0)}}.
\]
If $n_0$ is even, we set
\[
y_1=\sqrt[n_0]{-\dfrac{r_2(x_0)}{\bi_{n_0}(x_0)}} \ \ \text{and} \ \ y_2=-\sqrt[n_0]{-\dfrac{r_2(x_0)}{\bi_{n_0}(x_0)}},
\]
and the couples $(x_0,y_1),(x_0,y_2)\in L_A$, where $\dfrac{r_2(x_0)}{\bi_{n_0}(x_0)}\le0$. This case can happen if the above conditions hold, of course.
\item[(iii)] $\al_{m_0}(x_0)\neq0$, $\bi_{n_0}(x_0)=0$.

This case is similar to the previous case (ii).
\item[(iv)] $\al_{m_0}(x_0)\neq0$ and $\bi_{n_0}(x_0)\neq0$.

Through the equations of $(A)$ we get:
\[
y^{m_0}_0=-\frac{r_1(x_0)}{\al_{m_0}(x_0)} \ \ \text{and} \ \ y^{n_0}_0=-\frac{r_2(x_0)}{\bi_{n_0}(x_0)}.
\]
Through these equations we get:
\[
y^{n_0m_0}_0=(-1)^{n_0}\bigg(\frac{r_1(x_0)}{\al_{m_0}(x_0)}\bigg)^{n_0} \ \ \text{and} \ \ y^{n_0m_0}_0=(-1)^{m_0}\bigg(\frac{r_2(x_0)}{\bi_{n_0}(x_0)}\bigg)^{m_0}
\]
and by these equations we get:
\begin{align*}
(-1)^{n_0}\bigg(\frac{r_1(x_0)}{\al_{m_0}(x_0)}\bigg)^{n_0}&=(-1)^{m_0}
\bigg(\frac{r_2(x_0)}{\bi_{n_0}(x_0)}\bigg)^{m_0}\\
&\Leftrightarrow(-1)^{m_0-n_0}r_2(x_0)^{m_0}\al_{m_0}
(x_0)^{n_0}-r_1(x_0)^{n_0}\bi_{n_0}(x_0)^{m_0}=0.
\end{align*}
So, we solve this case as follows:

We find the real roots of polynomial $(-1)^{m_0-n_0}r_2(x)^{m_0}\al_{m_0}(x)^{n_0}-r_1(x)^{n_0}\bi_{n_0}(x)^{m_0}$, so that $\al_{m_0}(x)\neq0$ and $\bi_{n_0}(x)\neq0$.

For every $x_0$, we get $y_0$ so that $y^{m_0}_0=-\dfrac{r_1(x_0)}{\al_{m_0}(x_0)}$ as in the previous case (ii).
\end{enumerate}
\medskip

6) $\bi_{n_0}(x)\not\equiv0$, $\al_{v_0}(x)\equiv\bi_{\mi_0}(x)\equiv0$, $q_1(x,y),q_2(x,y)$ are two pure polynomials. So, we have to solve the system:
\[
\left\{\begin{array}{r}
         \al_{m_0}(x)y^{m_0}+q_1(x,y)=0 \\ [2ex]
         \bi_{n_0}(x)y^{n_0}+q_2(x,y)=0
       \end{array}\right.   \eqno{\mbox{(A)}}
\]
where $deg_yq_1(x,y)\ge1$, $deg_yq_2(x,y)\ge1$, and $q_1(x,y),q_2(x,y)$ are two monomials. So, in this case we have the system:
\[
\left\{\begin{array}{r}
         \al_{m_0}(x)y^{m_0}+\al_0(x)y^{\la_1}=0 \\ [2ex]
         \bi_{n_0}(x)y^{n_0}+\bi_0(x)y^{\la_2}=0
       \end{array}\right.,
\]
where $\al_0(x),\bi_0(x)$ are two polynomials so that one of them (at least) is non-zero and $\la_1,\la_2\in\N$, so that $\la_1<m_0$ and $\la_2<n_0$. We can write the system as follows:
\[
\left\{\begin{array}{r}
         y^{\la_1}(\al_{m_0}(x)y^{m_0-\la_1}+\al_0(x))=0 \\ [2ex]
         y^{\la_2}(\bi_{n_0}(x)^{n_0-\la_2}+\bi_0(x))=0
       \end{array}\right.,  \eqno{\mbox{(A)}}
\]
so for every $x\in\R$, the couple $(x,0)\in L_A$, which is false, because $L_A$ is finite. Thus, this cannot occur. \medskip

7) We suppose that $\bi_{n_0}(x)\not\equiv0$ $q_1(x,y)\equiv r_1(x)$, $\al_{v_0}(x)\not\equiv0$ or $\bi_{\mi_0}(x)\not\equiv0$, $q_2(x,y)\equiv r_2(x)$, where $r_1(x)\not\equiv0$ or $r_2(x)\not\equiv0$. So, we get the system:
\[
\left\{\begin{array}{r}
  \al_{m_0}(x)y^{m_0}+\al_{v_0}(x)y^{v_0}+r_1(x)=0 \\ [2ex]
  \bi_{n_0}(x)y^{n_0}+\bi_{\mi_0}(x)y^{\mi_0}+r_2(x)=0
\end{array}\right. \eqno{\mbox{(A)}}
\]
We distinguish some cases:
\begin{enumerate}
\item[(i)] $r_1(x)\equiv0$ and $r_2(x)\not\equiv0$. So, we get the system:
\[
\left\{\begin{array}{lr}
         \al_{m_0}(x)y^{m_0}+\al_{v_0}(x)y^{v_0}& =0 \\ [2ex]
                \bi_{n_0}(x)y^{n_0}+\bi_{\mi_0}(x)y^{\mi_0}+r_2(x)& =0
       \end{array}\right. \eqno{\mbox{(A)}}
\]
Let $(x_0,y_0)\in L_A$.

Through the first equation we get
\[
y^{v_0}(\al_{m_0}(x)y^{m_0-v_0}+\al_{v_0}(x))=0
\]
If $y=0$, by the second equation we get $r_2(x_0)=0$. So, if $x_0$ is a root of $r_2(x)$, then the couple $(x_0,0)\in L_A$.

Let $y_0\neq0$. Then we get:
\[
\al_{m_0}(x_0)y^{m_0-v_0}_0+\al_{v_0}(x_0)=0 \ \ \text{and} \ \ \bi_{n_0}(x_0)y^{n_0}_0+\bi_{\mi_0}(x_0)y^{\mi_0}_0+r_2(x_0)=0.
\]
If $r_2(x_0)=0$ we have some of the previous cases that we have already examined. So, we suppose that $r_2(x_0)\neq0$. If $n_0<m_0$ we have a system that we supposedly can solve through the induction step. Thus, we suppose that $n_0=m_0$. So we have the system:
\[
\left\{\begin{array}{lr}
         \al_{m_0}(x_0)y^{m_0-v_0}_0+\al_{v_0}(x_0) & =0 \\ [2ex]
                  \bi_{n_0}(x_0)y^{m_0}_0+\bi_{\mi_0}(x_0)y^{\mi_0}+r_2(x_0) & =0
       \end{array}\right.
\]
We will postpone this case, because we will examine a more general case later, that covers this case.
\item[(ii)] $r_1(x)\not\equiv0$ and $r_2(x)\equiv0$.

This case is similar to the previous.
\item[(iii)] $r_1(x)\not\equiv0$ and $r_2(x)\not\equiv0$.
\end{enumerate}

Let $(x_0,y_0)\in L_A$. If $r_1(x_0)=0$, or $r_2(x_0)=0$ we have some of the previous cases. So, we suppose that $r_1(x_0)\neq0$ and $r_2(x_0)\neq0$. We have some cases:
\begin{enumerate}
\item[(i)] $\al_{v_0}(x)\not\equiv0$ and $\bi_{\mi_0}(x)\equiv0$.

So, we have the system
\[
\left\{\begin{array}{lr}
        \al_{m_0}(x)y^{m_0}+\al_{v_0}(x)y^{v_0}  & +r_1(x)=0 \\ [2ex]
        \bi_{n_0}(x)y^{m_0}  & +r_2(x)=0
       \end{array}\right..  \eqno{\mbox{(A)}}
\]
We consider some cases:
\begin{itemize}
\item[(a)]\ \ Let $(x_0,y_0)\in L_A$, $\al_{m_0}(x_0)=\bi_{n_0}(x_0)=0$. We have examined this case previously.
\item[(b)] \ \ $\al_{m_0}(x_0)=0$ and $\bi_{n_0}(x_0)\neq0$. We have examined this case previously.
\item[(c)] \ \ $\al_{m_0}(x_0)\neq0$ and $\bi_{n_0}(x_0)=0$.

If $\al_{v_0}(x_0)=0$ we have examined this case previously. So, we suppose that $\al_{v_0}(x_0)\neq0$. Then we get $r_2(x_0)=0$ by the second equation of (A) because $\bi_{n_0}(x_0)=0$, that is false by our supposition. So, this case cannot occur.
\item[(d)] \ \ $\al_{m_0}(x_0)\neq0$ and $\bi_{n_0}(x_0)\neq0$.
\end{itemize}

If $\al_{v_0}(x_0)=0$, we have examined this case previously. So, we suppose that $\al_{v_0}(x_0)\neq0$. We will examine this case later.
\item[(ii)] $\al_{v_0}(x)\equiv0$ and $\bi_{\mi_0}(x)\not\equiv0$.

This case is similar to the previous:
\item[(iii)] $\al_{v_0}(x)\not\equiv0$ and $\bi_{\mi_0}(x)\not\equiv0$.

We have some cases here:
\begin{itemize}
\item[(a)] \ \ We suppose that $\bi_{n_0}(x_0)=\bi_{\mi_0}(x_0)=0$, and $\al_{m_0}(x_0)\neq0$, $\al_{v_0}(x_0)\neq0$. So, we have the system:
\[
\left\{\begin{array}{r}
         \al_{m_0}(x_0)y^{m_0}+\al_{v_0}(x_0)y^{v_0}+r_1(x_0)=0 \\  [2ex]
         r_2(x_0)=0
       \end{array}\right.. \eqno{\mbox{(A)}}
\]
This cannot happen because $r_2(x_0)\neq0$ by our supposition.
\item[(b)] \ \ We suppose that $\al_{m_0}(x_0)=\al_{v_0}(x_0)=0$. Then we get that $r_1(x_0)=0$, which is false by our supposition. Thus, this case cannot occur.
\item[(c)] \ \ If $\al_{m_0}(x_0)=0$, or $\al_{v_0}(x_0)=0$, or $\bi_{n_0}(x_0)=0$, or $\bi_{\mi_0}(x_0)=0$, then we get some of the previous cases.
\item[(d)] $\al_{m_0}(x_0)\neq0$ and $\al_{v_0}(x_0)\neq0$ and $\bi_{n_0}(x_0)\neq0$ and $\bi_{\mi_0}(x_0)\neq0$.

We have to solve the system:
\[
\left\{\begin{array}{r}
  \al_{m_0}(x_0)y^{m_0}+\al_{v_0}(x_0)y^{v_0}+r_1(x_0)=0 \\ [2ex]
  \bi_{n_0}(x_0)y^{m_0}+\bi_{\mi_0}(x_0)y^{\mi_0}+r_2(x_0)=0
\end{array}\right.. \eqno{\mbox{(A)}}
\]
We have here the basic case of this system.

We will examine this case later in a more general case.
\end{itemize}
\end{enumerate}

Now, we will examine the system:
\[
\left\{\begin{array}{r}
         \al_{m_0}(x)y^{m_0}+\al_{v_0}(x)y^{v_0}+q_1(x,y)=0 \\ [2ex]
         \bi_{n_0}(x)y^{n_0}+\bi_{\mi_0}(x)y^{\mi_0}+q_2(x,y)=0 \end{array}\right.,
\]
where $\al_{m_0}(x)\not\equiv0$, $m_0\ge3$, $v_0<m_0$, $n_0\le m_0$ $n_0>\mi_0$, $q_1(x,y),q_2(x,y)$ are two pure polynomials, $\al_{v_0}(x)\not\equiv0$, $\bi_{\mi_0}(x)\not\equiv0$.

We candistinguish the following cases: \medskip

1) $\bi_{n_0}(x)\equiv0$.

We get $(x_0,y_0)\in L_A$.

If $\al_{m_0}(x_0)=0$, then we have a system from the induction step. So, we suppose $\al_{m_0}(x_0)\neq0$.

We have some cases:
\begin{enumerate}
\item[(i)] $\al_{v_0}(x_0)=\bi_{\mi_0}(x_0)=0$.

Then, we analyze polynomials $q_1(x,y),q_2(x,y)$ and we reach a system of the following form
\[
\left\{\begin{array}{r}
  \al_{m_0}(x_0)y^{m_0}+\al_{v_1}(x_0)y^{v_1}+q_3(x_0,y_0)=0 \\ [2ex]
\bi_{m_1}(x_0)y^{\mi_1}+q_3(x_0,y_0)=0
\end{array}\right., \eqno{\mbox{(A)}}
\]
where $\al_{v_1}(x_0)\neq0$, $\bi_{\mi_1}(x_0)\neq0$, $v_1<m_0$, $\mi_1\in\N$, $deg_yq_3(x,y)<v_1$,\\ $deg_yq_3(x,y)<\mi_1$.

We will see later how we can solve such a system.
\item[(ii)] $\al_{v_0}(x_0)\neq0$ and $\bi_{\mi_0}(x_0)=0$.

Then, we analyze polynomial $q_2(x,y)$ and we reach a system of the following form:
\[
\left\{\begin{array}{rr}
         \al_{v_0}(x_0)y^{m_0}_0+\al_{v_0}(x_0)y^{v_0}_0+q_1(x_0,y_0)=0 & (1)\\ [2ex]
         \bi_{\mi_1}(x_0)y^{\mi_1}_0+r_2(x_0)=0 & (2)
       \end{array} \right., \eqno{\mbox{(B)}}
\]
If $\bi_{\mi_1}(x_0)=0$, we get from (2) that: $r_2(x_0)=0$.

In this case we solve system (B) as follows:

We take $x_0$ that is a common root of polynomials in the second equation (B) so that $\al_{m_0}(x_0)\neq0$ and $\al_{v_0}(x_0)\neq0$, and we find $y_0$ from the first equation of (B).
\item[(iii)] $\al_{v_0}(x_0)\neq0$ and $\bi_{\mi_0}(x_0)\neq0$. We will see later how we solve this system.\medskip

{\bf After all the above cases we reach now in the most important case}.\medskip

We have the system:
\[
\left\{\begin{array}{r}
         \al_{m_0}(x)y^{m_0}+\al_{v_0}(x)y^{v_0}+q_1(x,y)=0 \\ [2ex]
         \bi_{n_0}(x)y^{n_0}+\bi_{\mi_0}(x)y^{\mi_0}+q_2(x,y)=0 \end{array}\right.     \eqno{\mbox{(A)}}
\]
where we have:
\[
\begin{array}{l}
  m_0>v_0>deg_yq_1(x,y), \\ [2ex]
  n_0>\mi_0>deg_yq_2(x,y),
\end{array}
\]
$\al_{m_0}(x)\not\equiv0$, $\al_{v_0}(x)\not\equiv0$, $\bi_{n_0}(x)\not\equiv0$, $\bi_{\mi_0}(x)\not\equiv0$, $q_1(x,y)$, $q_2(x,y)$ be two pure polynomials.

We distinguish some cases:
\item[(i)] $m_0=n_0$.

We also have some cases here.
\begin{itemize}
\item[(a)] \ \ $v_0=\mi_0$.

So, we have the system:
\[
\left\{\begin{array}{r}
         \al_{m_0}(x)y^{m_0}+\al_{v_0}(x)y^{v_0}+q_1(x,y)=0 \\ [2ex]
         \bi_{m_0}(x)y^{m_0}+\bi_{v_0}(x)y^{v_0}+q_2(x,y)=0 \end{array}\right. \eqno{\mbox{(A)}}
\]
Firstly, we examine the case where:
\[
(x_0,y_0)\in L_A \ \ \text{and} \ \ \al_{m_0}(x_0)\cdot\al_{v_0}(x_0)\cdot\bi_{m_0}(x_0)\cdot\bi_{m_0}(x_0)
\cdot\bi_{v_0}(x_0)\neq0.
\]
Let
\[
D=\left|\begin{array}{cc}
          \al_{m_0}(x_0) & \al_{v_0}(x_0) \\ [2ex]
          \bi_{m_0}(x_0) & \bi_{v_0}(x_0)
        \end{array}\right|=\al_{m_0}(x_0)\bi_{v_0}(x_0)-\al_{v_0}(x_0)
        \bi_{m_0}(x_0).
\]
We suppose that $D\neq0$.

This exactly {\bf is the first basic case}. We will study three basic cases overall. We consider the linear system:
\[
\left\{\begin{array}{cc}
         \al_{m_0}(x_0)z+\al_{v_0}(x_0)\oo=-q_1(x_0,y_0) & (1) \\ [2ex]
         \bi_{m_0}(x_0)z+\bi_{v_0}(x_0)\oo=-q_2(x_0,y_0) & (2)
       \end{array}\right. \eqno{\mbox{(B)}}
\]
We set:
\[
D_1=\left|\begin{array}{lc}
            -\al_1(x_0,y_0) & \al_{v_0}(x_0) \\  [2ex]
            -q_2(x_0,y_0) & \bi_{v_0}(x_0)
          \end{array}\right| \ \ \text{and} \ \
D_2=\left|\begin{array}{lc}
            \al_{m_0}(x_0) & -q_1(x_0,y_0) \\  [2ex]
            \bi_{m_0}(x_0) & -q_2(x_0,y_0)
          \end{array}\right|.
\]
That is, we have:
\[
\begin{array}{l}
  D_1=\al_{v_0}(x_0)q_2(x_0,y_0)-\bi_{v_0}(x_0)q_1(x_0,y_0) \ \ \text{and} \\ [2ex]
  D_2=\bi_{m_0}(x_0)q_1(x_0,y_0)-\al_{m_0}(x_0)q_2(x_0,y_0). \end{array}
\]
\end{itemize}
\end{enumerate}

Because $D\neq0$, by our supposition, we take it that system (B) has only one solution $(z_0,\oo_0)$, where $z_0=\dfrac{D_1}{D}$ (3) and $\oo_0=\dfrac{D_2}{D}$ (4), as it is well known by linear algebra by Cramer's law.

Because $(x_0,y_0)\in L_A$, (by our supposition) this means that the couple $(y^{m_0}_0,y^{v_0}_0)$ is a solution of system (B).

But, $(z_0,\oo_0)$ is the unique solution of system (B). So, we have: $(z_0,\oo_0)=(y^{m_0}_0,y^{v_0}_0)\Leftrightarrow z_0=y^{m_0}_0$ (5) and $\oo_0=y^{v_0}_0$ (6). By (3), (4), (5) and (6), we get:
\[
y^{m_0}_0=\frac{D_1}{D} \ \ \text{(7)} \ \ \text{and} \ \ y^{v_0}_0
=\frac{D_2}{D}. \ \ \text{(8)}
\]
Now, we use the obvious relation of numbers $y^{m_0}_0$ and $y^{v_0}_0$, that is: $y^{m_0}_0=y^{m_0-v_0}_0\cdot y^{v_0}_0$ (9), where $v_0<m_0$, by our supposition.

Replacing by (7) and (8) in (9), we get:
\[
\frac{D_1}{D}=y^{m_0-v_0}_0\cdot\frac{D_2}{D}\Leftrightarrow D_2y^{m_0-v_0}_0-D_1=0. \eqno{\mbox{(10)}}
\]
From the above we see that $(x_0,y_0)$ satisfies the two equations:
\[
\left\{\begin{array}{rc}
         Dy^{v_0}_0-D_2=0 & (11) \\ [2ex]
         D_2y^{m_0-v_0}_0-D_1=0 & (12)
       \end{array}\right.  \eqno{\mbox{(C)}}
\]
We notice that polynomials in (11) and (12) have degree with respect to $y$ lower than $m_0$.\\
Let us consider now the following systems
\[
\left\{\begin{array}{lc}
         \al_{m_0}(x)y^{m_0}+\al_{v_0}(x)y^{v_0}+q_1(x,y)=0 & (13) \\ [1.5ex]
         \bi_{m_0}(x)y^{m_0}+\bi_{v_0}(x)y^{v_0}+q_2(x,y)=0 & (14) \\[1.5ex]
         y\cdot\al_{m_0}(x)\cdot\al_{v_0}(x)\bi_{m_0}(x)\cdot\bi_{v_0}(x)D\neq0 &
       \end{array}\right. \eqno{\mbox{(A)}}
\]
\[
\left\{\begin{array}{lc}
  Dy^{v_0}-D_2=0 & (15) \\ [1.5ex]
  D_2y^{m_0-v_0}-D_1=0 & (16) \\ [1.5ex]
  y\cdot\al_{m_0}(x)\cdot\al_{v_0}(x)\cdot\bi_{m_0}(x)\cdot\bi_{v_0}(x)
  \cdot D\neq0 &
\end{array}\right. \eqno{\mbox{(B)}}
\]
where $D=\al_{m_0}(x)\bi_{v_0}(x)-\al_{v_0}(x)\bi_{m_0}(x)$.
\[
\begin{array}{l}
  D_1=\al_{v_0}(x)q_2(x,y)-\bi_{v_0}(x)q_1(x,y), \\ [2ex]
  D_2=\bi_{m_0}(x)q_1(x,y)-\al_{m_0}(x)q_2(x,y).
\end{array}
\]
It is obvious that $deg_y(Dy^{v_0})=v_0<m_0$, because $D\neq0$ and $deg_y(D_2y^{m_0-v_0})<m_0$, because $deg_yD_2<v_0$, by our suppositions. We will prove now that $L_A=L_B$. It is obvious that $L_A\subseteq L_B$ (17) from the previous procedure, because we got equations (11) and (12) of system ($\Ga$) from equations of system A.

Now, let $(x_0,y_0)\in L_B$.

Through equations (15), (16) of (B) and the fact that $y_0\neq0$, we get:
\[
y^{v_0}_0=\frac{D_2}{D} \ \ \text{(18)} \ \ \text{and} \ \ y^{m_0-v_0}_0=\frac{D_1}{D_2} \ \ \text{(19)}
\]
Through equations (18) and (19) we get: $y^{m_0}_0=\dfrac{D_1}{D}$ (20). Now, we consider system (B). Because $D\neq0$ this system has unique solution $(z_0,\oo_0)=\Big(\dfrac{D_1}{D},\dfrac{D_2}{D}\Big)$ (21), from Cramer's Law. Through (18), (20) and (21) we get $z_0=y^{m_0}_0$ (22) and $\oo_0=y^{v_0}_0$ (23).

Replacing (22) and (23) in equations of (B) we take it that $(x_0,y_0)\in L_A$, so $L_B\subseteq L_A$ (24). By (17) and (24), we have $L_A=L_B$ (25). The equality (25) means that: in order to solve system (A), it suffices to solve system (B), whose degree with respect to $y$ is smaller than $m_0$, that is the degree of system (A) with respect to $y$. But with the induction step, we can solve a system whose degree with respect to $y$ is smaller than $m_0$, and thus we complete this case.
\smallskip

{\bf The second basic case is the following} $D\not\equiv0$, but
\[
D(x_0)=\al_{m_0}(x_0)\bi_{v_0}(x_0)-\al_{v_0}(x_0)\bi_{m_0}(x_0)=0.
\]
In this case the two equations of system (A) are equivalent to those of linear algebra, as we have shown in prerequisites.

So, we can solve this case as follows:

We find the roots of polynomial $D=\al_{m_0}(x)\bi_{v_0}(x)-\al_{v_0}(x)\bi_{m_0}(x)$, so that: $\al_{m_0}(x)\cdot\al_{v_0}(x)\cdot\bi_{n_0}(x)\cdot\bi_{\mi_0}(x)$.

For every such root $x_0$ we find $y_0$ from one of the equations of (A) that are equivalent. We can complete this case by finding the solutions of the form $(x,0)$ (if any).\medskip

{\bf Third basic case (singular case)}

We suppose that $D\equiv0\equiv\al_{m_0}(x)\bi_{v_0}(x)-\al_{v_0}(x)\bi_{m_0}(x)$. We call this case {\bf the singular case}.

We consider the system:
\[
\left\{\begin{array}{lc}
         \al_{m_0}(x)y^{m_0}+\al_{v_0}(x)y^{v_0}+q_1(x,y)=0 & (1) \\ [1.5ex]
         \bi_{m_0}(x)y^{m_0}+\bi_{v_0}(x)y^{v_0}+q_2(x,y)=0 & (2) \\ [1.5ex]
         \al_{m_0}(x)\al_{v_0}(x)\bi_{m_0}(x)\bi_{v_0}(x)\neq0 &
       \end{array}\right.. \eqno{\mbox{(A)}}
\]
We consider our general supposition. That is, we suppose $L_A\neq\emptyset$. Let $(x_0,y_0)\in L_A$. Then we get:
\[
\left\{\begin{array}{cc}
  \al_{m_0}(x_0)y^{m_0}_0+\al_{v_0}(x_0)y^{v_0}_0=-q_1(x_0,y_0) & (3) \\ [2ex]
  \bi_{m_0}(x_0)y^{m_0}_0+\bi_{v_0}(x_0)y^{v_0}_0=-q_2(x_0,y_0) & (4)
\end{array}\right.. \eqno{\mbox{(B)}}
\]
We get
\[
D(x_0)=\al_{m_0}(x_0)\bi_{v_0}(x_0)-\al_{v_0}(x_0)\bi_{m_0}(x_0)=0.
\]
Let
\[
D_1(x_0,y_0)=\al_{v_0}(x_0)q_2(x_0,y_0)-\bi_{v_0}(x_0)q_1(x_0.y_0)
\]
\[
D_2(x_0,y_0)=\bi_{m_0}(x_0)q_1(x_0,y_0)-\al_{m_0}(x_0)q_2(x_0,y_0).
\]

We now consider the following system
\[
\left\{\begin{array}{cc}
 \al_{m_0}(x_0)z+\al_{v_0}(x_0)\oo=-q_1(x_0,y_0)  & (5) \\ [2ex]
  \bi_{m_0}(x_0)z+\bi_{v_0}(x_0)\oo=-q_2(x_0,y_0) & (6)
\end{array}\right. \eqno{\mbox{($\Ga$)}}
\]
Through the previous system (B) we have that $(y^{m_0}_0,y^{v_0}_0)$ is a solution of ($\Ga$). That is, ($\Ga$) is a linear system that has a solution and $D(x_0)=0$. So, we have that\linebreak $D_1(x_0,y_0)=0$ through linear algebra.

We consider now the following two systems: (A) and the following:
\[
\left\{\begin{array}{lc}
        \al_{m_0}(x)y^{m_0}+\al_{v_0}(x)y^{v_0}+q_1(x,y)=0  & (7) \\ [1.5ex]
         \al_{v_0}(x)q_2(x,y)-\bi_{v_0}(x)q_1(x,y)=0 & (8) \\[1.5ex]
         y\cdot\al_{m_0}(x)\al_{v_0}(x)\bi_{m_0}(x)\bi_{v_0}(x)\neq0 &
       \end{array}\right. \eqno{\mbox{($\De$)}}
\]
From the above we have $L_A\subseteq L_\De$ (9).

We can now prove the reverse inclusion of (9).

Let $(x_0,y_0)\in L_\De$. Of course $(x_0,y_0)$ satisfies equation (1) of (A). We distinguish two cases: \medskip

(i) $q_1(x_0,y_0)\neq0$.

We can consider linear system ($\Ga$). This system has $D=D_1=0$ by our supposition $D\equiv0$ and $D_1=0$ because (8) holds for $(x_0,y_0)$, that is $D_1(x_0,y_0)=0$. So, system ($\Ga$) has an infinity of solutions and $D_2(x_0,y_0)=0$, because $D=D_1(x_0,y_0)=0$. Of course we have: $\al_{m_0}(x_0)q_{v_0}(x_0)\bi_{m_0}(x_0)\bi_{v_0}(x_0)\neq0$, by our supposition.

By relation $D(x_0)=0$ we take:
\[
\al_{m_0}(x_0)\bi_{v_0}(x_0)-\al_{v_0}(x_0)\bi_{m_0}(x_0)=0\Leftrightarrow
\frac{\al_{m_0}(x_0)}{\bi_{m_0}(x_0)}=\frac{\al_{v_0}(x_0)}{\bi_{v_0}(x_0)}. \eqno{\mbox{(10)}}
\]
By equation $D_1(x_0,y_0)=0$ we take:
\[
\al_{v_0}(x_0)q_2(x_0,y_0)=\bi_{ v_0}(x_0)q_1(x_0,y_0)\Rightarrow
\frac{\al_{v_0}(x_0)}{\bi_{v_0}(x_0)}=\frac{q_1(x_0,y_0)}{q_2(x_0,y_0)} \eqno{\mbox{(11)}}
\]
We have $q_2(x_0,y_0)\neq0$ or else if $q_2(x_0,y_0)=0\Rightarrow q_1(x_0,y_0)=0$, that is false by our supposition. So, (11) holds. By (10) and (11) we set
\[
0\neq\la=\frac{\al_{m_0}(x_0)}{\bi_{m_0}(x_0)}=\frac{\al_{v_0}(x_0)}{\bi_{v_0}(x_0)}=
\frac{q_1(x_0,y_0)}{q_2(x_0,y_0)}\Rightarrow\bi_{m_0}(x_0)=\frac{1}{\la}\al_{m_0}(x_0), \eqno{\mbox{(12)}}
\]
\[
\bi_{v_0}(x_0)=\frac{1}{\la}\al_{v_0}(x_0),  \eqno{\mbox{(13)}}
\]
\[
q_2(x_0,y_0)=\frac{1}{\la}q_1(x_0,y_0).  \eqno{\mbox{(14)}}
\]
By (12), (13) and (14) we get:
\begin{align*}
\bi_{m_0}(x_0)y^{m_0}_0+\bi_{v_0}(x_0)y^{v_0}_0+q_2(x_0,y_0)&=
\frac{1}{\la}\al_{m_0}(x_0)y^{m_0}_0+\frac{1}{\la}\al_{v_0}(x_0)y^{v_0}_0+\frac{1}{\la}
q_1(x_0,y_0)\\
&=\frac{1}{\la}(\al_{m_0}(x_0)y^{m_0}_0+\al_{v_0}(x_0)y^{v_0}_0+q_1(x_0,y_0)\\
&=\frac{1}{\la}\cdot0=0,
\end{align*}
because $(x_0,y_0)\in L_\De$, which means that $(x_0,y_0)$ satisfies equality (7). So, we proved that if $(x_0,y_0)\in L_\De$ and $q_1(x_0,y_0)\neq0$, then $(x_0,y_0)\in L_A$.\medskip

(ii) $q_1(x_0,y_0)=0$.

Then, because $(x_0,y_0)\in L_\De$, through equality (8) we get: $q_2(x_0,y_0)=0$, because $\al_{v_0}(x_0)\neq0$, by our supposition.

As previously, because $D(x_0)=0$ and $\bi_{m_0}(x_0)\bi_{v_0}(x_0)\neq0$, we take it that: (12), and (13) holds, so
\begin{align*}
\bi_{m_0}(x_0)y^{m_0}_0+\bi_{v_0}(x_0)y^{v_0}_0+q_2(x_0,y_0)
&=\frac{1}{\la}\al_{m_0}(x_0)y^{m_0}_0+\frac{1}{\la}\al_{v_0}(x_0)y^{v_0}_0+0\\
&=\frac{1}{\la}(\al_{m_0}(x_0)y^{m_0}_0+\al_{v_0}(x_0)y^v_0+q_1(x_0,y_0)=0,
\end{align*}
by equality (7) of $(\De)$ because $(x_0,y_0)\in L_\De$ by our supposition.

So, equality (2) of (A) holds, that is $(x_0,y_0)\in L_A$. So, we have $L_\De\subseteq L_A$ (15). Through (9) and (15), we get: $L_A=L_\De$. So, in order to solve system (A), it suffices to solve system $(\De)$. What is the profit from system $(\De)$. The profit is that polynomial in equation (8) of $(\De)$, that is $D_1$, has $deg_yD_1(x,y)<v_0$ or $D_1(x,y)\equiv0$. We examine now how we exploit these facts.

We leave the case $D_1(x,y)\equiv0$, for the end.

We examine now the case where $D_1(x,y)\not\equiv0$. We can write $D_1(x,y)$ in the following form:
\[
D_1(x,y)=\al_{v_1}(x)y^{v_1}+\al_{v_2}(x)y^{v_2}+q_3(x,y),
\]
where $v_0>v_1>v_2$, $deg_yq_3(x,y)<v_2$, or $q_3(x,y)\equiv0$. This is the general case.

We suppose, also, that $\al_{v_1}(x)\not\equiv0$, $\al_{v_2}(x)\not\equiv0$ and $q_3(x,y)$ is a pure polynomial.

We get:
\[
y^{m_0-v_1}D_1(x,y)=\al_{v_1}(x)y^{m_0}+\al_{v_2}(x)y^{m_0-v_1+v_2}+y^{m_0-v_1}q_3(x,y),
\]
where $deg_y(y^{m_0-v_1}q_3(x,y))<m_0-v_1+v_2$ because $deg_yq_3(x,y)<v_2$, by our supposition.

We consider the system:
\[
\left\{\begin{array}{l}
         \al_{m_0}(x)y^{m_0}+\al_{v_0}(x)y^{v_0}+q_1(x,y)=0 \\ [1.5ex]
         \al_{v_1}(x)y^{m_0}+\al_{v_2}(x)y^{m_0-v_1+v_2}+y^{m_0-v_1}q_3(x,y)=0 \\ [1.5ex]
         y\al_{m_0}(x)\al_{v_0}(x)\bi_{m_0}(x)\bi_{v_0}(x)\neq0
       \end{array}\right.. \eqno{\mbox{(E)}}
\]
If $\al_{v_1}(x)=0$, or $\al_{v_2}(x)=0$ for $x\in\R$ we examine whether system (E) has a root of $\al_{v_1}(x)$ or $\al_{v_2}(x)$ that satisfies system (E). So, we examine the case where\linebreak $\al_{v_1}(x)\cdot\al_{v_2}(x)\neq0$.

Let
\[
D=\left|\begin{array}{cc}
          \al_{m_0}(x) & \al_{v_0}(x) \\ [2ex]
          \al_{v_1}(x) & \al_{v_2}(x)
        \end{array}\right|=\al_{m_0}(x)\al_{v_2}(x)-\al_{v_0}(x)\al_{v_1}(x).
\]
Then, system (E) is a system similar to system (A).

So, we examine the similar cases with the same way.

Here, we examine only the case where $D=\al_{m_0}(x)\al_{v_2}(x)-\al_{v_0}(x)\al_{v_1}(x)\equiv0$.\\
In this case we again reach a system similar to $(\De)$, so that the respective equation (8) of the new system has $D_1(x,y)\not\equiv0$.

We handle this case as follows:

In system (A) of page 23, we can take any of two equations in order to get an equivalent system as $(\De)$. So it helps us to take the equation in which the respective pure polynomial $q_1(x,y)$ or $q_2(x,y)$ has the smallest number of terms. For this reason in system (E) (that is similar to A) we take as a first equation (of system $(\De)$) the second equation because  this polynomial $y^{m_0-v_1}q_3(x,y)$ has at most $v_2$ terms with respect to $y$ (by its definition), where $v_2<v_1<v_0\Rightarrow v_2\le v_0-2$.

In the new system $(\De)$ we take that the respective $D_1(x,y)$ polynomial of $(\De)$ has $deg_yD_1(x,y)<v_0$, so, if we write this polynomial again in the form:
\[
D_1(x,y)=\al_{v_1}(x)y^{v_1}+\al_{v_2}(x)y^{v_2}+q_4(x,y),
\]
the new polynomial $q_4(x,y)$, has at most $v_2\le v_0-2$ terms.

So, the profit, is that the new polynomials $q_1(x,y)$, $q_2(x,y)$ of the new system $(\De)$ will have at most $v_0-2$ terms each one of them and the respective new polynomial $D_1(x,y)$ also. So, the profit is the following:

In system $(\De)$ polynomial $D_1(x,y)$ has at most $v_0$ terms with respect to $y$, whereas in a new system like $(\De)$ in a following stage the respective polynomial $D_1(x,y)$ of the new system $(\De)$ will have at most $v_0-2$ terms with respect to $y$.

With the same procedure we can see that the terms of the respective polynomials $D_1(x,y)$ are decreasing, so that after a finite number of steps we reach a polynomial $D_1(x,y)\equiv0$ or $D_1(x,y)\equiv r(x)$ for  polynomial $r(x)\not\equiv0$. If $D_1(x,y)\equiv0$ we solve this case in the final step, or else if $r(x)\not\equiv0$, it suffices to find the roots of polynomial $r(x)$, otherwise we have some of the previous cases that we have already examined.

Now we will examine the remaining case. In system (A), page 21. If $v_0\neq\mi_0$ and $n_0=m_0$ we have the first basic case where $D(x_0)\neq0$.

Now, let $m_0\neq n_0$, that is $n_0<m_0$. If $n_0\ge v_0$, we have the first basic case. So, we can examine the case $v_0>n_0$. In this case we have:
\begin{align*}
&y^{m_0-n_0}(\bi_{n_0}(x)y^{n_0}+\bi_{\mi_0}(x)y^{\mi_0}+q_2(x,y))=0\\
&\Leftrightarrow\bi_{n_0}(x)y^{m_0}+\bi_{\mi_0}(x)y^{m_0-n_0+\mi_0}+y^{m_0-n_0}
q_2(x,y)=0
\end{align*}
and instead of (A) we examine the system:
\[
\left\{\begin{array}{l}
         \al_{m_0}(x)y^{m_0}+\al_{v_0}(x)y^{v_0}+q_1(x,y)=0 \\ [2ex]
         \bi_{n_0}(x)y^{m_0}+\bi_{\mi_0}(x)y^{m_0-n_0+\mi_0}+y^{m_0-n_0}q_2(x,y)=0 \end{array}\right.
\]
This system is of the case where $m_0=n_0$, which we have already examined. So, up to now, we have examined all the possible cases of the initial system except one only, that we will examine now.

In the third basic case we will examine now the case where $D_1(x,y)\equiv0$. Then, as in pages 23, 24 we take it that $D_2(x,y)\equiv0$, also that for every $(x,y)\in\R^2$ there exists  $c\in\R$, such that
\[
\al_{m_0}(x)y^{m_0}+\al_{v_0}(x)y^{v_0}+q_1(x,y)=c\cdot(\bi_{m_0}(x)y^{m_0}+\bi_{v_0}(x)
y^{v_0}+q_2(x,y)) \eqno{\mbox{($\ast$)}}
\]
and $c\neq0$. The number $c$ depends on the couple $(x,y)$, so it is better to write $c(x,y)$, instead of $c$.

Now, we will consider system (A$^\ast$)
\[
\left\{\begin{array}{l}
         \al_{m_0}(x)y^{m_0}+\al_{v_0}(x)y^{v_0}+q_1(x,y)=0 \\ [2ex]
         \al_{m_0}(x)\al_{v_0}(x)\bi_{m_0}(x)\bi_{v_0}(x)\neq0
       \end{array}\right.. \eqno{\mbox{(A$^\ast$)}}
\]
Equality $(\ast)$ gives us that
\[
L_A=L_{A^\ast}.
\]
So, in order to solve system (A) it suffices to solve the ``simpler'' system (A$^\ast$) that has only one equation.

Now, it is the time to exploit the unique supposition that we have not used up to now.

That is: The set $L_A=L_{A^\ast}$ is finite. As we have seen in the prerequisites there are polynomials $p(x,y)$ of two real variables that have a finite set of roots only. \\
For example let:
\[
p(x,y)=(x^2-4)^2+(y^2-9)^2.
\]
It is easy to see that
\[
L_{p(x,y)}=\{(2,3),(2,-3),(-2,3),(-2,-3)\}.
\]
We denote:
\[
R(x,y)=\al_{m_0}(x)y^{m_0}+\al_{v_0}(x)y^{v_0}+q_1(x,y),
\]
for simplicity.

So, we solve the system:
\[
\left\{\begin{array}{l}
         R(x,y)=0 \\ [2ex]
         \al_{m_0}(x)\al_{v_0}(x)\bi_{m_0}(x)\bi_{v_0}(x)\neq0
       \end{array}\right.. \eqno{\mbox{(A$^\ast$)}}
\]
Of course, we get $R(x,y)\not\equiv 0$, because $\al_{m_0}\not\equiv0$.

Now, it is the time to use the results of our prerequisites.

By the suppositions of the third case we get $\al_{m_0}(x)\not\equiv0$ and $m_0>1$, that gives that $R(x,y)$ is a pure polynomial, that has a finite set of roots, non empty.

We apply Corollary 3.16 by our prerequisites and we take it that 0 is the global maximum or minimum of $R(x,y)$.

Without loss of generality we suppose that 0 is the global minimum of $R(x,y)$. This means that if we consider the function $F:U\ra\R$ (where\\ $U=\{(x,y)\in\R^2\mid\al_{m_0}(x)\al_{v_0}(x)\bi_{m_0}(x)\bi_{v_0}(x)\neq0\}$ is an open subset of $\R^2$).\linebreak $F((x,y))=R(x,y)$ for every $(x,y)\in U$, then it holds $F((x,y))\ge0$ for every $(x,y)\in U$, and there exists  $(x_0,y_0)\in U$ so that $F((x_0,y_0))=0$.

Let $(x_0,y_0)\in\R^2$ so that $(x_0,y_0)\in L_{A^\ast}$. Then, we have $F((x_0,y_0))=0$ and function $F$ has a global minimum in $(x_0,y_0)$. Then, by Theorem 3.17 we get $\nabla F(x_0,y_0)=(0,0)$. So we have: $\dfrac{\partial F}{\partial y}((x_0,y_0))=0$.

We can now consider the system:
\[
\left\{\begin{array}{l}
         F((x,y))=0 \\ [1.5ex]
         \dfrac{\partial F}{\partial y}((x,y))=0 \\ [1.5ex]
         \al_{m_0}(x)\al_{v_0}(x)\bi_{m_0}(x)\bi_{v_0}(x)\neq0
       \end{array}\right.. \eqno{\mbox{((A$_1$)}}
\]
Of course we get $L_{A_1}\subseteq L_{A^\ast}=L_A$ and by the above we also get: $L_A^\ast\subseteq L_{A_1}$. So we get:
\[
L_{A_1}=L_A.
\]
So, in order to solve system (A$^\ast$) it suffices to solve system (A$_1$). We need to write a more analytic system (A$_1$). We get:
\[
\left\{\begin{array}{l}
         \al_{m_0}(x)y^{m_0}+\al_{v_0}(x)y^{v_0}+q_1(x,y)=0 \\ [1.5ex]
         m_0\al_{m_0}(x)y^{m_0-1}+v_0\al_{v_0}(x)y^{v_0-1}+
         \dfrac{\partial q_1}{\partial y}(x,y)=0 \\ [1.5ex]
         \al_{m_0}(x)\al_{v_0}(x)\bi_{m_0}(x)\bi_{v_0}(x)\neq0
       \end{array}\right.  \eqno{\mbox{(A$_1$)}}
\]
Because $m_0>v_0\Rightarrow m_0-1\ge v_0$. This shows that system (A$_1$) is the first basic case, and so we can transfer system (A$_1$) to a system that has smaller than $m_0$ degree with respect to $y$ that we can solve with the induction step. So, inductively we have managed to solve the initial system in any case. So, we have completed our second stage.
\subsection{Third stage} 
Let a polynomial
\[
p(z)=\al_0+\al_1z+\cdots+\al_{v-1}z^{v-1}+\al_vz^v,
\]
for $v\in\N$, $\al_1\in\C$, for $i=0,1,\ld,v$, $\al_v\neq0$, of one complex variable.

We are now ready to solve completely the equation $p(z)=0$, or in other words to find the roots of polynomial $p(z)$ with degree $v$.

We distinguish two cases: \medskip

(i) $\al_i\in\R$ for every $i=0,1,\ld,v$, and (ii) $\al_i\in\C$, $i=0,1,\ld,v$. Firstly, we prove the following lemma:
\setcounter{lem}{3}
\begin{lem}\label{lem2.4} (A well known lemma).

Let $p(z)$, be a polynomial as above with degree $v=deg p(z)\in\N$. Then, there exist two polynomials $p_1(x,y)$, $p_2(x,y)$ of two real variables with real coefficients, so that it holds:
\[
p(x+yi)=p_1(x,y)+ip_2(x,y)
\]
for every $(x,y)\in\R^2$.
\end{lem}
\begin{proof}
We can prove this lemma with induction above the degree $v$ of $p(z)$. Let $p(z)=\al_0+\al_1z$, $\al_0,\al_1\in\R$, $\al_1\neq0$. Let $(x,y)\in\R^2$. We get:
\[
p(x+yi)=\al_0+\al_1(x+yi)=(\al_0+\al_1x)+\al_1yi, \ \ \text{for} \ \ v=1
\]
so for $p_1(x,y)=\al_0+\al_1x$ and $p_2(x,y)=\al_1y$, the result holds.\smallskip

For $v=2$.\smallskip\\
Let $p(z)=\al_0+\al_1z+\al_2z^2$, where $\al_0,\al_1,\al_2\in\R$, $\al_2\neq0$.

Let $z=x+yi$, $(x,y)\in\R^2$. We get:
\begin{align*}
p(z)&=p(x+yi)=\al_0+\al_1(x+yi)+\al_2(x+yi)^2\\
&=(\al_0+\al_1+\al_2x^2-\al_2y^2)+(\al_1y+2\al_2xy)i,
\end{align*}
so for $p_1(x,y)=\al_0+\al_1x+\al_2x^2-\al_2y^2$ and $p_2(x,y)=\al_1y+2\al xy$, the result holds. We suppose now, that the result holds for any $1\le i\le k_0\in\N$. We can prove that result holds for $k_0+1$.

Let
\[
p(z)=\al_0+\al_1z+\cdots+\al_{n_0}z^{k_0}+\al_{k_0+1}z^{k_0+1},
\]
be a polynomial with $\al_{k_0+1}\neq0$, $\al_i\in\R$, for every $i=0,1,\ld,k_0+1$.

Let $(x,y)\in\R^2$. We have:
\[
p(z)=q(z)+\al_{k_0+1}z^{k_0+1},
\]
we distinguish two cases:\medskip

(a) $q(z)\not\equiv0$. Then, through the induction step we can show that there exist two polynomials $p_1(x,y)$, $p_2(x,y)$ of two real variables $x$ and $y$ with real coefficients, so that:
\[
q(x+y_i)=p_1(x,y)+p_2(x,y)i \ \ \text{for every} \ \ (x,y)\in\R^2. \eqno{\mbox{(1)}}
\]
We get:
\begin{align*}
\al_{k_0+1}z^{k_0+1}=&\,\al_{k_0+1}(x+yi)^{k_0+1}=\al_{k_0+1}\sum^{k_0+1}_{j=0}
\binom{k_0+1}{j}x^j\cdot(yi)^{k_0+1-j}\\
=&\,\al_{k_0+1}\sum^{k_0+1}_{j=0}\binom{k_0+1}{j}x^jy^{k_0+1-j}i^{k_0+1-j} \\
=&\sum_{\scriptsize{\begin{array}{c}
          k_0+1-j=2\rho \\
          \rho\in \N \\
          \al<j\le k_0+1
        \end{array}}}\al_{k_0+1}\binom{k_0+1}{j}x^jy^{k_0+1-j}(-1)^{(n_0+1-j)/2}\\
        &+\sum_{\scriptsize{\begin{array}{c}
          k_0+1-j=2\rho+1 \\
          \rho\in \N \\
          0\le j\le k_0+1
        \end{array}}}\al_{k_0+1}x^jy^{k_0+1-j}i^{k_0+1-j}\\
=&\,q_1(x,y)+iq_2(x,y)   \hspace*{6.5cm} {\mbox{(2)}}
\end{align*}
where
\[
q_1(x,y)=\sum_{\scriptsize{\begin{array}{c}
          k_0+1-j=2\rho \\
          \rho\in \N \\
          0\le j\le k_0+1
        \end{array}}}\al_{k_0+1}\binom{k_0+1}{j}x^jy^{k_0+1-j}(-1)^{(n_0+1-j)/2} \ \ \text{and}
\]
\[
iq_2(x,y)=\sum_{\scriptsize{\begin{array}{c}
          k_0+1-j=2\rho+1 \\
          \rho\in \N \\
          0\le j\le k_0+1
        \end{array}}}\al_{k_0+1}x^jy^{k_0+1-j}i^{k_0+1-j}
\]
where $q_1(x,y)$, $q_2(x,y)$ are two polynomials of the two real variables with real coefficients because $i^{2v+1}=i$ or $-i$, $v\in\N$.

So, we get: $\al_{k_0+1}z^{k_0+1}=q_1(x,y)+iq_2(x,y)$. So, we get: by (1) and (2)
\begin{align*}
p(z)&=q(z)+\al_{k_0+1}z^{k_0+1}=(p_1(x,y))+p_2(x,y)i)+(q_1(x,y)+q_2(x,y)i) \\
&=(p_1(x)y)+q_1(x,y))+(p_2(x,y)+q_2(x,y))i
\end{align*}
and the result also holds for every $(x,y)\in\R^2$.

(b) $q(z)\equiv0$. Then, with the above equality (2) we get\\ $p(z)=\al_{k_0+1}z^{k_0+1}=q_1(x,y)+iq_2(x,y)$ for every $(x,y)\in\R^2$ and the result also holds. So, by induction we see that the result holds in this case.

Now, we suppose that $\al_i\in\C$, for every $i=0,1,\ld,v$.

Let $\al_j=\bi_j+\ga_ji$ for every $j=0,1,\ld,v$, where $\bi_j,\ga_j\in\R$ for every $j=0,1,\ld,v$. Let $z=x+yi\in\C$, $(x,y)\in\R^2$. We get:
\begin{align*}
p(z)&=p(x+yi)=\al_0+\al_1z+\cdots+\al_{v-1}z^{v-1}+\al_vz^v \\
&=(\bi_0+\ga_0i)+(\bi_1+\ga_1i)z+\cdots+(\bi_{v-1}+\ga_{v-1}i)z^{v-1}
+(\bi_v+\ga_vi)z^v \\
&=(\bi_0+\bi_1z+\cdots+\bi_{v-1}z^{v-1}+\bi_vz^v)+(\ga_0+\ga_1z+\cdots+
\ga_{v-1}z^{v-1}+\ga_vz^v)i.  \hspace*{0.2cm} (3)
\end{align*}
In the previous case (i) we see that there exist polynomials $p_1(x,y)$, $p_2(x,y)$, $q_1(x,y)$, $q_2(x,y)$ of the two real variables $x$ and $y$ with real coefficients, so that
\[
\bi_0+\bi_1z+\cdots+\bi_{v-1}z^{v-1}+\bi_vz^v=p_1(x,y)+p_2(x,y)i   \ \eqno{\mbox{(4)}}
\]
and
\[
\ga_0+\ga_1z+\cdots+\ga_{v-1}z^{v-1}+\ga_vz^v=q_1(x,y)+q_2(x,y)i \eqno{\mbox{(5)}}
\]
for every $(x,y)\in\R^2$.

By (3), (4) and (5) we get:
\begin{align*}
p(z)&=(p_1(x,y)+p_2(x,y)i)+(q_1(x,y)+q_2(x,y)i)i \\
&=(p_1(x,y)-q_2(x,y))+(p_2(x,y)+q_1(x,y))i
\end{align*}
and the result holds also. \qb \medskip

With the help of this Lemma we can now solve the equation $p(z)=0$ as follows:

We examine the general case where
\[
p(z)=\al_0+\al_1z+\cdots+\al_{v-1}z^{v-1}+\al_vz^v, \ \ v\in\N, \ \ \al_v\neq0, \ \ \al_i\in\C,
\]
for every $i=0,1,\ld,v$.

With the help of the above lemma we write:
\[
p(x+yi)=q_1(x,y)+q_2(x,y)i \eqno{\mbox{($\ast$))}}
\]
for every $(x,y)\in\R^2$, where $q_1(x,y),q_2(x,y)$ are two polynomials of two real variables $x$ and $y$ with real coefficients.

Let $A$ be the set of roots of $p(z)$. We consider the system:
\[
\left\{\begin{array}{l}
         q_1(x,y)=0 \\ [2ex]
         q_2(x,y)=0
       \end{array}\right.. \eqno{\mbox{(B)}}
\]
It is obvious from the above equality $(\ast)$ that $A=L_B$. So, in order to find all the roots of A, it suffices to find all the real roots of system (B).

So, we solve system B with the method we have developed in the second stage, and thus we find all the roots of polynomial $p(z)$.

Our method has been completed now because our supposition (S) (that system (B) has a solution) is satisfied because the same holds for (A). So, in all the cases we can reduce our initial system to a system in which the two polynomials have a lower degree than that of the polynomials of the initial system. Thus, we apply the induction step and the system is solved inductively.
\end{proof}
\section{Prerequisites}\label{sec3}
\noindent

a) Prerequisites from Algebra.\medskip

We use some basic tools and results from theory of polynomials.

We denote $\C[z]$ as the set of complex polynomials. We denote $\R[x]$ as the set of real polynomials, that is the set of polynomials of one real variable with coefficients in the set of real numbers $\R$.

We begin with the following basic result, that is a simple implication of the algorithm of Euclidean division.
\begin{prop}\label{prop3.1}
Let $p(z)\in\C[z]$, $deg p(z)\ge1$. The number $r\in\C$ is a root of $p(z)$ if and only if there exists a unique polynomial $q(z)\in\C[z]$ so that:
\[
p(z)=(z-r)q(z). \quad \blacksquare
\]
\end{prop}

We need the definition of multiplicity of a root of a polynomial.
\begin{Def}\label{def3.1}
Let $p(z)\in\C[z]$. Let $\rho\in\C$ be a root of $p(z)$. The natural number $m$ is a multiplicity of the root $\rho$ of $p(z)$ if polynomial $(z-\rho)^m$ divides $p(z)$, whereas polynomial $(z-\rho)^{m+1}$ does not divide $p(z)$.
\end{Def}

As consequence of Proposition 3.1 there is the following proposition:
\begin{prop}\label{prop3.3}
Every root of a polynomial $p(z)\in\C[z]$ has a multiplicity, that is unique.  \qb
\end{prop}

We state now the fundamental Theorem of algebra, whose proof is not simple and needs some tools from analysis.
\begin{thm}\label{thm3.4}
Every complex polynomial $p(z)$, with $deg p(z)\ge1$ has at least one root. \qb
\end{thm}

From Theorem 3.4 and Proposition 3.1 we get the following fundamental result:
\begin{thm}\label{thm3.5}
Let $p(z)\in\C[z]$ be a complex polynomial with $deg p(z)\ge1$. Then $p(z)$ has a finite number of different roots.

Let $\rho_1,\rho_2,\ld,\rho_v$ be the different roots of $p(z)$ with respective to multiplicities $m_1,m_2,\ld,m_v$. Then, the following formula holds:
\[
p(z)=\al\cdot(z-\rho_1)^{m_1}(z-\rho_2)^{m_2}\cdots (z-\rho_v)^{m_v},
\]
where $\al\neq0$ and $\al$ is the coefficient of the monomial of greater grade $m_0=deg p(z)$, and $m_0=m_1+m_2+\cdots+m_v$. \qb
\end{thm}

Now, we describe a simple algorithm in order to find the multiplicity of a root of a complex polynomial.\medskip\\
{\bf 3.6. An algorithm for the multiplicity of a root}\medskip

Let $p(z)\in\C[z]$ be a complex polynomial of degree $deg p(z)\ge1$.

By Theorem \ref{thm3.5} polynomial $p(z)$ has a finite number of roots. Let $\rho$ be a root of $p(z)$. We describe with details a way in order to find the multiplicity of $\rho$.

By Proposition 3.1 there exists a unique polynomial $q(z)$ so that:
\[
p(z)=(z-\rho)q(z).   \eqno{\mbox{(1)}}
\]
We find the polynomial $q(z)$ through the algorithm of Euclidean division, for example using Horner's scheme.

Afterwards, we compute the number $q(\rho)$, for example with Horner's scheme.\\
If $q(\rho)\neq0$, then the root $\rho$ has multiplicity 1. In order to prove this we suppose that the root $\rho$ does not have multiplicity 1. By Proposition 3.3 the root $\rho$ has a unique multiplicity, $m\in\N$ (see Definition 3.2). Because of $m\neq1$, we have that $m\ge2$. By the definition of multiplicity we have that polynomial $(z-\rho)^m$ divides $p(z)$. This means (by the definition of division) that there exists a polynomial $R(z)\in\C[z]$ sο that:
\[
p(z)=(z-\rho)^mR(z).  \eqno{\mbox{(2)}}
\]
By relations (1) and (2) we get:
\[
(z-\rho)q(z)=(z-\rho)^mR(z)\Leftrightarrow(z-\rho)(q(z)-(z-\rho)^{m-1}R(z)=0. \eqno{\mbox{(3)}}
\]
The expressions $z-\rho$ and $q(z)-(z-\rho)^{m-1}R(z)$ are polynomials in $\C[z]$ of course, because $m\ge2$ (as we have seen). Because of $z-\rho\not\equiv0$, we take it that
\[
q(z)-(z-\rho)^{m-1}R(z)=0, \eqno{\mbox{(4)}}
\]
because the Ring of polynomials $\C[z]$ is an integer neighbourhood, as is well known from Algebra. Relation (4) gives $q(\rho)=0$ (because $m\ge2$), which is false because we have supposed that $q(\rho)\neq0$. So, if $q(\rho)\neq0$, then root $\rho$ has multiplicity 1.

Whereas if $q(\rho)=0$, then through Proposition \ref{prop3.1}, we take it that there exists a polynomial $q_1(z)\in\C[z]$, so that:
\[
q(z)=(z-\rho)q_1(z).  \eqno{\mbox{(5)}}
\]
By (1) and (5) we take that
\[
p(z)=(z-\rho)^2q_1(z).  \eqno{\mbox{(6)}}
\]
Relation (6) tells us that polynomial $(z-\rho)^2$ divides $p(z)$. Afterwards, we find polynomial $q_1(z)$ by (5) with the Euclidean Algorithm, for example from Horner's scheme, because we have found polynomial $q(z)$ previously. After that, we compute number $q_1(\rho)$, for example with Horner's scheme. If $q_1(\rho)\neq0$, then the multiplicity of $\rho$ is 2, with a proof similar to what we had found previously. Or otherwise if $q_1(\rho)=0$, then again through Proposition \ref{prop3.1} there exists a polynomial $q_2(z)\in\C[z]$ so that
\[
q_1(z)=(z-\rho)q_2(z).  \eqno{\mbox{(7)}}
\]
Through (6) and (7) we take it:
\[
p(z)=(z-\rho)^3q_2(z).  \eqno{\mbox{(8)}}
\]
We inductively continue this procedure of finding a sequence of polynomials %
\[
q_j(z)\in\C[z],
\]
for $j=1,2,\ld$, where $q_j(z)=(z-\rho)q_{j+1}(z)$, for $j=1,2,\ld$, and
\[
p(z)=(z-\rho)^{j+1}q_j(z).
\]
If $p(z)=(z-\rho)^{j+1}q_j(z)$ for $j\in\N$, (where $q_j(\rho)\neq0)$, then the multiplicity of $\rho$ is $j+1$, with a proof similar to what we have shown previously. This procedure stops if some natural number $j\in\N$, or if $deg p(z)=v_0\in\N$, then we take it that $p(z)=(z-\rho)^{v_0+1}q_{v_0}(z)$, where $q_{v_0}(z)\neq0$ (or else $p(z)=0$, which is false because $deg p(r)\ge1$, by supposition), so $deg((z-\rho)^{v_0+1}q_{v_0}(z))\ge v_0+1$, which is false of course because $deg p(z)=v_0$.

That is we take it that $p(z)=(z-\rho)^{j+1}q_j(z)$ for some $j\in\N$, $j<v_0-1$, and $q_j(\rho)\neq0$, that gives that multiplicity of $\rho$ is $j+1<v_0$ otherwise we take it that
\[
p(z)=(z-\rho)^{v_0}q_{v_0-1}(z).  \eqno{\mbox{(9)}}
\]
Relation (9) gives that $q_{v_0-1}(z)\neq0$, (or else $p(z)=0$ which is false of course) and by relation (9) we take it also that $q_{v_0-1}(z)$ is a constant polynomial with value say $c_0$. That is $p(z)=(z-\rho)^{v_0}c_0$. Of course polynomial $(z-\rho)^{v_0+1}$ cannot divide $p(z)$, because this polynomial has a degree $deg((z-\rho)^{v_0+1})>v_0=deg p(z)$, that gives that multiplicity of $\rho$ is $v_0$ (by the definition of multiplicity). So we have described a complete algorithm that gives us the multiplicity of a root of a complex\linebreak polynomial. \qb\smallskip\\
\noindent
{\bf Remark 3.7}
{\em We can combine Proposition \ref{prop3.1} with Theorem 3.4 and the previous algorithm (and of course Proposition \ref{prop3.3}), in order to prove Theorem \ref{thm3.5}. We leave it as an easy exercise for the reader. So far we have developed all we need from polynomials of one complex variable. We also obtain some basic results from Linear Algebra. Here we will now consider the following linear system of two equations
\[
\left\{\begin{array}{lc}
         \al_1x+\bi_1y=\ga_1 & (1) \\ [2ex]
         \al_2x+\bi_2y=\ga_2 & (2)
       \end{array}\right. \eqno{\mbox{(A)}}
\]
where $\al_i,\bi_i,\ga_i\in\C$ for $i=1,2$. We consider the determinants $D,D_x,D_y$ where
\[
\begin{array}{l}
  D=\left|\begin{array}{cc}
            \al_1 & \bi_1 \\ [1.5ex]
            \al_2 & \bi_2
          \end{array}\right|=\al_1\bi_2-\al_2\bi_1,
   \\ [4ex]
  D_x=\left|\begin{array}{cc}
              \ga_1 & \bi_1 \\[1.5ex]
              \ga_2 & \bi_2
            \end{array}\right|=\ga_1\bi_2-\bi_1\ga_2,D_y
=\left|\begin{array}{cc}
\al_1 & \ga_1 \\ [1.5ex]
\al_2 & \ga_2
\end{array}\right|=\al_1\ga_2-\al_2\ga_1.
\end{array}
\]
When $D\neq0$, then system (A) has only one solution $(x_0,y_0)$, where $x_0=\dfrac{D_x}{D}$, $y_0=\dfrac{D_y}{D}$. When $D=0$ and $D_x\neq0$, or $D_y\neq0$, then system (A) does not have any solution, whereas when $D=D_x=D_y=0$, then system (A) has an infinite number of solutions except only in the case where $\al_1=\al_2=\bi_1=\bi_2=0$ and only one of the numbers $\ga_1,\ga_2$ is non zero. We need the case where $D\neq0$ and the case where $D=D_x=D_y=0$. We consider the case where $D=D_x=D_y=0$. We suppose that system (A) is a pure system of two variables $x$ and $y$, that is, we suppose that at least one of the numbers $\al_1,\al_2$ is non-zero also. That is $\al_1\neq0$ or $\al_2\neq0$ and $\bi_1\neq0$ or $\bi_2\neq0$, otherwise we do not have a system of equations of two different variables}.\medskip

We have two cases: \medskip

(i) One from the six numbers $\al_i,\bi_i,\ga_i$, $i=1,2$ is zero

Let $\al_1=0$ (3). We have $D=0$, that is $\al_1\bi_2-\al_2\bi_1=0\Rightarrow\al_1\bi_2=\al_2\bi_1$ (4). Through (3) and (4) we have $\al_2\bi_1=0$ (5). Because of $\al_1=0$ and our hypothesis we have $\al\neq0$ (6). By (5) and (6) we get $\bi_1=0$ (7).

Through (3) and the fact that $D_y=0$ we get in a similar way that $\ga_1=0$. That is the equation (1) is the equation $0\cdot x+0\cdot y=0$, with set of solutions the set $\R^2$. This means that system (A) is equivalent to the equation (2) of (A) only. If $\bi_1=0$, or $\ga_1=0$, we get in a similar way that $\al_1=\bi_1=\ga_1=0$ and we have similarly the same implication, that is system (A) is equivalent with equation (2) of (A) only. If $\al_2=0$, or $\bi_2=0$, or $\ga_2=0$, we take it that $\al_2=\bi_2=\ga_2=0$ in an analogous way and finally system (A) is equivalent with the equation (1) of (A) only.

Now, we suppose that $\al_1\bi_1\ga_1\al_2\bi_2\ga_2\neq0$, that is, non of six numbers $\al_i,\bi_i,\ga_i$, $i=1,2,$ is zero. We have $D=0\Leftrightarrow\al_1\bi_2-\al_2\bi_1=0\Leftrightarrow\dfrac{\al_2}{\al_1}
=\dfrac{\bi_2}{\bi_1}$ $(\al_1\neq0$, $\bi_1\neq0$). We set $\la=\dfrac{\al_2}{\al_1}=\dfrac{\bi_2}{\bi_1}$ (8).

We have $D_x=0\Leftrightarrow\ga_1\bi_2-\bi_1\ga_2=0\Leftrightarrow\la=\dfrac{\bi_2}{\bi_1}
=\dfrac{\ga_2}{\ga_1}$ (9). With (8) and (9) we get: $\al_2=\la\al_1$, $\bi_2=\la\bi_1$, $\la_2=\la\ga_1$, that is $\al_2x+\bi_2y=\ga_2\Leftrightarrow(\la\al_1)x+(\la\bi_1)y=(\la\ga_1)\Leftrightarrow
\la\cdot(\al_1x+\bi_1y)=\la\ga_1\overset{\la\neq0}{\Leftrightarrow}\al_1x+\bi_y=\ga_1$, that is equations (1) and (2) of (A) are equivalent, that is they have the same set of solutions, which means that system (A) is equivalent to one only from the equations (1) and (2), whichever of the two).

So, we have proved that in the case of $D=D_x=D_y=0$, system (A) has an infinite number of solutions and it is equivalent with only one from the equations (1) and (2). So we have stated our prerequisites from Algebra.\medskip \\
{\bf b) Prerequisites from Analysis}\medskip

As it is well known, by Galois theory, there are no formulas that give the roots of an arbitrary polynomial as a function of its coefficients with radicals. So, for an arbitrary polynomial the only way to find its roots is to approximate them with a numerical method. Perhaps, the simplest numerical method for algebraic equations is the bisection method, which is presented in all classical books of Numerical Analysis:

It is a simple method, and here we have based in it in our problem. The bisection method has very weak suppositions, and it is convenient for secondary students also.

Let $\al,\bi\in\R$, $\al<\bi$, and $f:[\al,\bi]\ra\R$ be a continuous function. We suppose that $f(\al)\cdot f(\bi)<0$. Then, function $f$ has a root, at least in the interval $(\al,\bi)$, and bisection method approximates a root of $f$ in $(\al,\bi)$, as closely as we want with a specific minor error.

There are many different numerical methods that find the roots in a specific interval. We will not discuss this subject. This is a vast subject in Numerical Analysis. In this text, it is enough for us to find only one root in a specific interval and approximate it using bisection method.

The solution to all the real roots of a polynomial will be based on the following basic lemma.\medskip\\
\noindent
{\bf Basic Lemma 3.8}. {\em Let $v\in\N$, $v\ge3$, $p(x)=\al_vx^v+\al_{v-1}x^{v-1}+\cdots+\al_1x+\al_0$, be a polynomial $p(x)\in\R[x]$, with degree $deg p(x)=v$.

Let $\rho_1,\rho_2,\ld,\rho_k$ be all the different real roots of polynomial $p'(x)$, $k\in\N$, $k\ge2$, $\rho_i\neq\rho_j$, for all $i,j\in\{1,2,\ld,k\}$, $i\neq j$.

Then, we can find, with an algorithm, all the real roots of $p(x)$, with their multiplicities.}
\begin{proof}
Let $L=\{\rho_1,\rho_2,\ld,\rho_k\}$, be the set of all real roots of $p'(x)$. We suppose, also, without loss of generality that $\rho_1<\rho_2<\cdots<\rho_k$.

Let $i_0\in\{1,\ld,k-1\}$. Then $p'(x)>0$ for every $x\in(\rho_{i_0},\rho_{i_0+1})$ or $p'(x)<0$ for every $x\in(\rho_{i_0},\rho_{i_0+1})$. This gives that $p$ is a strictly decreasing or strictly increasing function on $[\rho_{i_0},\rho_{i_0+1}]$. If $p(\rho_{i_0})=0$, then $\rho_{i_0}$ is the unique root of $p$ in $[\rho_{i_0},\rho_{i_0+1}]$. The same holds if $p(\rho_{i_0+1})=0$, that is $\rho_{i_0+1}$ is the unique root of $p$ in $[\rho_{i_0},\rho_{i_0+1}]$, if $p(\rho_{i_0+1})=0$.

Of course polynomial $p$ can't have the numbers $\rho_{i_0}$ and $\rho_{i_0+1}$ as roots simultaneously, by its monotonicity. We suppose now that $p(\rho_{i_0})\cdot p(\rho_{i_0+1})\neq0$. Then, if $p(\rho_{i_0})\cdot p(\rho_{i_0+1})>0$, polynomial $p$ does not have any root in $[\rho_{i_0},\rho_{i_0+1}]$. If $p(\rho_{i_0})\cdot p(\rho_{i_0+1})<0$, then $p$ has one root exactly in the interval $[\rho_{i_0},\rho_{i_0+1}]$, and more specifically this root belongs in $(\rho_{i_0},\rho_{i_0+1})$.

Applying the bisection method we find this root, because the suppositions of bisection method are satisfied now. We do the same in every interval $[\rho_i,\rho_{i+1}]$.

So we find all the roots of $p$ in the interval $[\rho_1,\rho_k]$. We examine the roots in $[\rho_k,+\infty)$. Because $\al_v\neq0$, we have two cases:
\begin{enumerate}
\item[i)] If $\al_v>0$, then $\lim\limits_{x\ra+\infty}p(x)=+\infty$.

Then $p$ is a strictly increasing function in $[\rho_k,+\infty)$.
\begin{itemize}
\item[a)]
$p(\rho_k)=0$, then $\rho_k$ is the unique root of $p$ in $[\rho_k,+\infty)$.
\item[b)] If $p(\rho_k)>0$, then $p$ does not have any root in $[\rho_k,+\infty)$.
\item[c)] If $p(\rho_k)<0$, then $p$ has one root exactly, (say $\rho_{k+1})$ in $[\rho_k,+\infty)$ and more specifically $\rho_{k+1}\in(\rho_k,+\infty)$.
\end{itemize}

Because $\lim\limits_{x\ra+\infty}p(x)=+\infty$, there exists some $x_0\in\R$, $x_0>\rho_k$, so that $p(x_0)>0$. Then $p(\rho_k)\cdot p(x_0)<0$ and thus $\rho_{k+1}\in(\rho_k,x_0)$.

Applying bisection method in $[\rho_{k+1},x_0]$, we approximate the root $\rho_{k+1}$. Later, we will see how we compute a number like $x_0$, in order to apply bisection method.
\item[ii)] If $\al_v<0$, then $\lim\limits_{x\ra+\infty}p(x)=-\infty$. Polynomial $p$ is a strictly decreasing function in $[\rho_k,+\infty)$.
\begin{itemize}
\item[a)] If $p(\rho_k)=0$, then $\rho_k$ is the unique root of $p$ in $[\rho_k,+\infty)$.
\item[b)] If $p(\rho_k)<0$, then $p$ does not have any root in $[\rho_k,+\infty)$.
\item[c)] If $p(\rho_k)>0$, then $p$ has unique one root in $[\rho_k,+\infty)$ (say $\rho_{k+1})$ and more specifically $\rho_{k+1}\in(\rho_k,+\infty)$.
\end{itemize}

Because $\lim\limits_{x\ra+\infty}p(x)=-\infty$, there exists some $x_0\in(\rho_k,+\infty)$, so that $p(x_0)<0$.

Then, $p(\rho_k)\cdot p(x_0)<0$, and $\rho_{k+1}\in(\rho_k,x_0)$, and applying bisection method we approximate the unique root $\rho_{k+1}$ in $(\rho_k,x_0)$. Now we examine the roots in $(-\infty,\rho_1]$. Whether $p(\rho_1)=0$, then $\rho_1$ is the unique root of $p$ in $(-\infty,\rho_1]$.

Now we suppose that $p(\rho_1)\neq0$. We examine two cases:
\item[i)] $\lim\limits_{x\ra-\infty}p(x)=+\infty$.

This happens when $v$ is even and $\al_v>0$, or $v$ is odd and $\al_v<0$. Then $p$ is a strictly decreasing function in $(-\infty,\rho_1]$.

i), 1) If $p(\rho_1)>0$, then $p$ does not have any root in $(-\infty,\rho_1]$.

i), 2) If $p(\rho_1)<0$, then $p$ has a unique root in $(-\infty,\rho_1]$ (say $\rho_{k+2})$ and more specifically $\rho_{k+2}\in(-\infty,\rho_1)$. Because $\lim\limits_{x\ra-\infty}p(x)=+\infty$, there exists some\linebreak $x_0<\rho_1$, so that $p(x_0)>0$. Then $\rho_{k+2}\in(x_0,\rho_1)$ and applying bisection method in $[x_0,\rho_1]$, we approximate root $\rho_{k+2}$.
\item[ii)]  $\lim\limits_{x\ra-\infty}p(x)=-\infty$. This is happened when $v$ is even and $\al_v<0$, or $v$ is odd and $\al_v>0$. Then $p$ is a strictly increasing function in $(-\infty,\rho_1]$.

    We have two cases:

    ii), 1)$p(\rho_1)<0$. Then, $p$ does not have any root in $(-\infty,\rho_1]$.

    ii), 2) $p(\rho_1)>0$. Then $p$ has unique root in $(-\infty,\rho_1]$ (say $\rho_{k+2})$ and more specifically $\rho_{k+2}\in(-\infty,\rho_1)$. Because $\lim\limits_{x\ra-\infty}p(x)=-\infty$, there exists some $x_0<\rho_1$, such that $p(x_0)<0$. Then $p(x_0)\cdot p(\rho_1)<0$ and $\rho_{k+2}\in(x_0,\rho_1)$.

    Applying bisection method in $[x_0,\rho_1]$ we approximate root $\rho_{k+2}$. All the implications of this lemma are easy to prove and are left as an easy exercise for the interested reader. The proofs are of secondary school.
\end{enumerate}
\end{proof}
\noindent
{\bf Corollary 3.9.} {\em
Basic Lemma 3.8 holds again, in the case when polynomial $p'$ has only one root}.
\begin{proof}
The proof is similar to that of basic lemma for the intervals $(-\infty,\rho_1]$ and $[\rho_1,+\infty)$, where $p'(\rho_1)=0$. \qb
\end{proof}
\noindent
{\bf Corollary 3.10.} {\em
Let $v\in\N$, $v\ge3$, $p(x)=\al_vx^v+\al_{v-1}x^{v-1}+\cdots+\al_1x+\al_0$, be a polynomial $p(x)\in\R[x]$, with degree $deg p(x)=v$.

We suppose that $p'$ does not have any root. Then $p$ has unique real root and we can construct an algorithm in order to find it}.
\begin{proof}
Of course $p'$ is a polynomial of even degree $deg p'=v-1$, so $p$ is a polynomial of odd degree. Thus $p$ has, at least, one real root. Because $p'$ does not have any root, we have $p'(x)\neq0$, for every $x\in\R$. Thus, $p'(x)>0$ for every $x\in\R$, or $p'(x)<0$ for every $x\in\R$, or else if there exist $\al,\bi\in\R$, so that $p'(\al)<0$ and $p'(\bi)>0$, (of course $\al\neq\bi$), then because $p'$ is a continuous function (as a polynomial) and $p'(\al)\cdot p'(\bi)<0$, we take it that there exists $\ga\in(\al,\bi)$ (if $\al<\bi$) or $\ga\in(\bi,\al)$  (if $\bi<\al$) so that $p'(\ga)=0$, that is a contradiction because $p'(x)\neq0$ for every $x\in\R$. Thus, $p$ is a strictly increasing function in $\R$, if $p'(x)>0$, for every $x\in\R$, or else $p$ is a strictly decreasing function in $\R$ if $p'(x)<0$ for every $x\in\R$. If $p$ is a strictly increasing function, then $\lim\limits_{x\ra+\infty}p(x)=+\infty$ and $\lim\limits_{x\ra-\infty}p(x)=-\infty$, or else if $p$ is a strictly decreasing function in $\R$, then $\lim\limits_{x\ra+\infty}p(x)=-\infty$ and $\lim\limits_{x\ra-\infty}p(x)=+\infty$.

Polynomial $p$ is a strictly increasing function if $\al_v>0$, or else if $\al_v<0$, then $p$ is a strictly decreasing function.

If $\al_v>0$, then because $\lim\limits_{x\ra+\infty}p(x)=+\infty$, there exists $y_0\in\R$, so that\linebreak $p(y_0)>0$, and because $\lim\limits_{x\ra-\infty}=-\infty$, there exists $x_0\in\R$, $x_0<y_0$, so that $p(x_0)<0$. So $p(x_0)\cdot p(y_0)<0$ and $p$ has unique root in $\R$, (say $\rho$) so that $\rho\in(x_0,y_0)$.

If $\al_v<0$, then because $\lim\limits_{x\ra+\infty}p(x)=-\infty$ there exists $y_0\in\R$, so that $p(y_0)<0$. Because $\lim\limits_{x\ra-\infty}p(x)=+\infty$, there exists $x_0\in\R$, $x_0<y_0$ so that $p(x_0)>0$. Thus $p(x_0)\cdot p(y_0)<0$, and $p$ has unique root in $\R$ (say $\rho$) so that $\rho\in(x_0,y_0)$.

We sill see later how we compute numbers $x_0,y_0$ as above.

In any of the cases above we apply the bisection method in the interval $[x_0,y_0]$, to find the unique real root of $p$. \qb
\end{proof}
\noindent
{\bf Remark 3.11}
{\em
The multiplicity of a root is found with the algebraic algorithm 3.6.

However, we can find the multiplicity of a root in an analytic way.

More specifically:\\
Let $p(x)\in\C[x]$ be a polynomial and $\rho$ be a root of $p$, where $deg p(x)=v\in\N$.

Then, there exists a unique natural number $k\in\N\cup\{0\}$ $k\le v-1$, so that:\\
$p(\rho)=0$, $p'(\rho)=0,\ld$, $p^{(k)}(\rho)=0$ and $p^{(k+1)}(\rho)\neq0$, that is $p^{(i)}(\rho)=0$, for all $i=0,1,\ld,k$ and $p^{(k+1)}(\rho)=0$, where $p^{(0)}(\rho)=p(\rho)$.

The natural number $k+1$ is the multiplicity of root $\rho$ of $p$. (Of course we have always $p^{(v)}(\rho)\neq0$)\\
This is a classical result in calculus, that is proven easily.

Now, we cover the gap from basic Lemma 3.8, computing a number like $x_0$ in this lemma}. \medskip\\
\noindent
{\bf Remark 3.12.} {\em
Let $p(x)\in\R[x]$ be a real polynomial,
\[
p(x)=\al_0+\al_1x+\cdots+\al_{v-1}x^{v-1}+\al_vx^v,\ \  v=deg p(x),\ \  v\ge3.
\]
We suppose that $\al_v>0$, and that $p'(x)$ has real roots}.
\begin{proof}
Let $\rho$ be the greater real root of $p'(x)$. We suppose that $p(\rho)<0$. We consider an arbitrary real number $x_0$, so that $x_0>\rho$, $x_0>1$ and\\
$x_0>\dfrac{|\al_0|+|\al_1|+\cdots+|\al_{v-1}|}{\al_v}$. We prove that $p(x_0)>0$.

We have of course $-|y|\le y$ for every $y\in\R$ (1). We apply (1) for

\[
y=\frac{\al_0}{x^v}+\frac{\al_1}{x^{v-1}}+\cdots+\frac{\al_{v-1}}{x}\eqno{\mbox{(1)}}
\]
for some $x\in\R-\{0\}$ and we have:
\[
-\bigg|\frac{\al_0}{x^v}+\frac{\al_1}{x^{v-1}}+\cdots+\frac{\al_{v-1}}{x}\bigg|
\le\frac{\al_0}{x^v}+\frac{\al_1}{x^{v-1}}+\cdots+\frac{\al_{v-1}}{x}.   \eqno{\mbox{(2)}}
\]
Adding the number $\al_v$ in two members of (2) we take:
\[
\al_v-\bigg|\frac{\al_0}{x^v}+\frac{\al_1}{x^{v-1}}+\cdots+\frac{\al_{v-1}}{x}\bigg|
\le\al_v+\frac{\al_0}{x^v}+\frac{\al_1}{x^{v-1}}+\cdots+\frac{a_{v-1}}{x}. \eqno{\mbox{(3)}}
\]
By the triangle inequality we take for $x>0$:
\begin{align*}
\bigg|\frac{\al_0}{x^v}+\frac{\al_1}{x^{v-1}}+\cdots+\frac{\al_{v-1}}{x}\bigg|
&\le\bigg|\frac{\al_0}{x^v}\bigg|+\bigg|\frac{\al_1}{x^{v-1}}\bigg|+\cdots+
\bigg|\frac{\al_{v-1}}{x}\bigg|\\
&\Leftrightarrow-\bigg(\frac{|\al_0|}{x^v}+\frac{|\al_v|}
{x^{v-1}}+\cdots+\frac{|\al_{v-1}|}{x}\bigg)\\
&\le-\bigg|\frac{\al_0}{x^v}+\frac{\al_1}{x^{v-1}}+\cdots+\frac{\al_{v-1}}{x}\bigg|  \hspace*{3cm} {\mbox{(4)}}
\end{align*}
Adding the number $\al_v$ in two members of (4) we take
\[
\al_v-\bigg(\frac{|\al_0|}{x^v}+\frac{|\al_1|}{x^{v-1}}+\cdots+\frac{|\al_{v-1}|}{x}\bigg)
\le\al_v-\bigg|\frac{\al_0}{x^v}+\frac{\al_1}{x^{v-1}}+\cdots+\frac{\al_{v-1}}{x}\bigg|,
\ \ \text{for} \ \ x>0.  \eqno{\mbox{(5)}}
\]
Let some $x>1$. Then we have:
\begin{align*}
&x\ge x,x^2\ge x,\ld,x^v\ge x\Rightarrow\frac{1}{x}\le\frac{1}{x}, \frac{1}{x^2}<\frac{1}{x},\ld,\frac{1}{x^v}<\frac{1}{x}\\
&\Rightarrow\frac{|\al_{v-1}|}{x}\le\frac{|\al_{v-1}|}{x},\frac{|\al_{v-2}|}{x^2}\le, \frac{|\al_{v-2}|}{x},\ld,
\frac{|\al_1|}{x^{v-1}}\le\frac{|\al_1|}{x},\frac{|\al_0|}{x^v}\le\frac{|\al_0|}{x}.
\end{align*}
Adding the above inequalities in pairs we take:
\begin{align*}
&\frac{|\al_{v-1}|}{x}+\cdots+\frac{|\al_1|}{x^{v-1}}+\frac{|\al_0|}{x^v}\le
\frac{|\al_0|+|\al_1|+\cdots+|\al_{v-1}|}{x}\Rightarrow\\
&-\frac{|\al_0|+|\al_1|+\cdots+|\al_{v-1}|}{x}\le-\bigg(\frac{|\al_{v-1}|}{x}+\cdots
+\frac{|\al_1|}{x^{v-1}}+\frac{|\al_0|}{x^v}\bigg)\Rightarrow \\
&\al_v-\frac{|\al_0|+|\al_1|+\cdots+|\al_{v-1}|}{x}\le\al_v-\bigg(\frac{|\al_{v-1}|}{x}
+\cdots+\frac{|\al_1|}{x^{v-1}}+\frac{|\al_0|}{x^v}\bigg). \hspace*{1.6cm} {\mbox{(6)}}
\end{align*}
Through inequalities (3), (5) and (6), we get:
\[
\al_v-\frac{|\al_0|+|\al_1|+\cdots+|\al_{v-1}|}{x}\le\al_v+\frac{\al_0}{x^v}+
\frac{\al_1}{x^{v-1}}+\cdots+\frac{\al_{v-1}}{x} \ \ \text{for} \ \ x>1.  \eqno{\mbox{(7)}}
\]
Now, for every $x>1$, $x>\dfrac{|\al_0|+|\al_1|+\cdots+|\al_{v-1}|}{\al_v}$, we take
\[
\al_v>\frac{|\al_0|+|\al_1|+\cdots+|\al_{v-1}|}{x}\Rightarrow\al_v-\frac{|\al_0|
+|\al_1|+\cdots+|\al_{v-1}|}{x}>0.  \eqno{\mbox{(8)}}
\]
Through (7) and (8), we take it that for every $x>1$, $x>\dfrac{|\al_0|+|\al_1|+\cdots+|\al_{v-1}|}{\al_v}$ we get
\[
\al_v+\dfrac{\al_{v-1}}{x}+\cdots+\frac{\al_1}{x^{v-1}}+\frac{\al_0}{x^v}>0.  \eqno{\mbox{($\ast$)}}
\]
This gives
\[
x^v\cdot\bigg(\al_v+\frac{\al_v-1}{x}+\cdots+\frac{\al_1}{x^{v-1}}+\frac{\al_0}{x^v}\bigg)>0\Leftrightarrow
p(x)>0 \ \ \text{(by the definition of $p(x)$)}  \eqno{\mbox{(9)}}
\]
We apply (9) for the number $x_0\in\R$ so that $x_0>\rho$, $x_0>1$ and\\
$x_0>\dfrac{|\al_0|+|\al_1|+\cdots+|\al_{v-1}|}{\al_v}$ and we take it that $p(x_0)>0$.

Thus, we have $p(\rho)\cdot p(x_0)<0$. This means that the unique real root of $p(x)$ in $[\rho,+\infty)$ belongs in $(\rho,c_0)$. Applying the bisection method in the interval $[\rho,x_0]$, we compute the unique real root $x_0^\ast$ of $p(x)$ in $[\rho,+\infty)$ that is $x^\ast_0\in(\rho,x_0)$. Of course if $p(\rho)>0$, polynomial $p$ does not have any real root in $[\rho,+\infty)$ as we have seen in basic Lemma 3.8 and if $p(\rho)=0$, then $\rho$ is the unique real root of $p$ in $[\rho,+\infty)$.

Now, we suppose that $\al_v<0$, and that $p'(x)$ has real roots.

Let $\rho$ be the greatest real root of $p'(x)$. We suppose that $p(\rho)>0$. We consider an arbitrary real number $x_0$, so that $x_0>\rho$, $x_0>1$ and
\[
x_0>\frac{|\al_0|+|\al_1|+\cdots+|\al_{v-1}|}{-\al_v}=\frac{|\al_0|+|\al_1|
+\cdots+|\al_{v-1}|}{|\al_v|}.  \eqno{\mbox{(10)}}
\]
By (10) we get (because $|\al_v|>0$ and $x>0$)
\[
|\al_v|>\frac{|\al_0|+|\al_1|+\cdots+|\al_{v-1}|}{x_0}\Rightarrow
\al_v+\frac{|\al_0|+|\al_1|+\cdots+|\al_{v-1}|}{x_0}<0.  \eqno{\mbox{(11)}}
\]
Let $x>1$. Because of $x>1$ we get
\begin{align*}
&x\ge x,x^2\ge x,\ld,x^{v-1}\ge x,x^v\ge x\Rightarrow\frac{1}{x^v}\le\frac{1}{x},\frac{1}{x^{v-1}}\le\frac{1}{x},\ld,
\frac{1}{x}\le\frac{1}{x}\Rightarrow \\
&\frac{|\al_0|}{x^v}\le\frac{|\al_0|}{x},\frac{|\al_1|}{x^{v-1}}\le\frac{|\al_1|}
{x^{v-1}},\ld,\frac{|\al_{v-1}|}{x}\le\frac{|\al_{v-1}|}{x}.
\end{align*}
Adding by pairs the previous inequalities we get:
\begin{align*}
&\frac{|\al_0|}{x^v}+\frac{|\al_1|}{x^{v-1}}+\cdots+\frac{|\al_{v-1}|}{x}\le
\frac{|\al_0|}{x}+\frac{|\al_1|}{x}+\cdots+\frac{\al_{v-1}|}{x}\Rightarrow \\
&\al_v+\frac{|\al_0|}{x^v}+\frac{\al_v|}{x^{v-1}}+\cdots+\frac{|\al_0-1|}{x}\le\al_v+
\frac{|\al_0|+|\al_1|+\cdots+|\al_{v-1}|}{x}. \hspace*{1.6cm} {\mbox{(12)}}
\end{align*}

Of course we get:
\begin{align*}
\al_0\le|\al_0|,\al_1\le|\al_1|,\ld,\al_{v-1}&\le|\al_{v-1}|\overset{x>0}{\Longrightarrow}
\frac{\al_0}{x^v}\le\frac{|\al_0|}{x^v},\frac{\al_1}{x^{v-1}}\\
&\le\frac{|\al_1|}{x^{v-1}}
,\ld,\frac{\al_{v-1}}{x}\le\frac{|\al_{v-1}|}{x}
\end{align*}
and adding by pairs the previous inequalities we get:
\begin{align*}
&\frac{\al_0}{x^v}+\frac{\al_1}{x^{v-1}}\cdots+\frac{\al_{v-1}}{x}\le\frac{|\al_0|}{x^v}
+\frac{|\al_1|}{x^{v-1}}+\cdots+\frac{|\al_{v-1}|}{x}\Rightarrow \\
&\al_v+\frac{\al_0}{x^v}+\frac{\al_1}{x^{v-1}}+\cdots+\frac{\al_{v-1}}{x}\le
\al_v+\frac{|\al_0|}{x^v}+\frac{|\al_1|}{x^{v-1}}+\cdots+\frac{|\al_{v-1}|}{x}.
 \hspace*{1.7cm} {\mbox{(13)}}
\end{align*}
Of course as we have seen in basic Lemma 3.8 if $p(\rho)<0$, then $p$ does not have any root in $[\rho,+\infty)$, and if $p(\rho)=0$, number $\rho$ is the unique real root of $p$ in $[\rho,+\infty)$. So far we have seen how we compute the unique real root of $p$ (if any) in $[\rho,+\infty)$, when $\rho$ is the greatest real root of $p'$.

In a similar way we compute the unique real root of $p$ in $(-\infty,\rho^\ast]$ (if any), where $\rho^\ast$ is the smallest real root of $p'$. It suffices to observe the following:

We simply write $p(x)=p(-(-x))$ and we find easily a polynomial $q\in\R[x]$, such that $p(x)=q(-x)$ (it is trivial to find such a polynomial $q$).

Now, for $x_0>1$, $x_0>\rho$ and $x_0>\dfrac{|\al_0|+|\al_1|+\cdots+|\al_{v-1}|}{|\al_v|}$ the previous inequalities (11), (12) and (13) hold simultaneously for $x=x_0$, thus we get
\begin{align*}
&\al_v+\frac{\al_{v-1}}{x_0}+\cdots+\frac{\al_1}{x^{v-1}_0}+\frac{\al_0}{a^v_0}<0\Rightarrow \\
&x^v_0\bigg(\al_v+\frac{\al_{v-1}}{x_0}+\cdots+\frac{\al_1}{x^{v-1}_0}+\frac{\al_0}{x^v_0}\bigg)
<0\Rightarrow p(x_0)<0 \ \ \text{(by the definition of $p$)}.
\end{align*}
So, we get $p(\rho)\cdot p(x_0)<0$ and $p$ has a root exactly in $(\rho,x_0)$, (say $x^\ast_0$), where $x^\ast_0$ is the unique real root of $p$ in $[\rho,+\infty)$.

We apply the bisection method in the interval $[\rho,x_0]$ and we compute the unique real root $x^\ast_0\in(\rho,x_0)$ of $p$ in $[\rho,+\infty)$.

Now we suppose that $p$ has real roots, and let $\rho$ be the smallest real root of $p$. Let $r$ be an arbitrary real root of $p$. Then $p(r)=0$ and by (14) we take that $q(-r)=0$, that is $-r$ is a real root of $q$. This gives that $-\rho$ is the biggest real root of $q$.

By (14) we also get that $p'(x)=-q'(-x)$, which gives that if $\al$ is a real root of $p'$, then $-\al$ is a real root of $q'$. Thus, because by supposition $p'$ has real roots, the same holds for $q'$. Then, we apply the previous procedure and we compute the greatest real root of $q$, (say $-\rho$), which means that $\rho$ is the smallest real root of $p$. Thus, we compute the smallest real root of $p$ also in any case. Now, we consider a polynomial $p\in\R[x]$, with $deg p(x)=v\in\N$, $v\ge3$, so that polynomial $p'$ does not have any real root. Because of $p'$ does not have any root and $deg p'(x)\ge2$ (because $deg p(x)\ge3$), we conclude that $p'$ is a polynomial of even degree (because any polynomial of odd degree has a real root at least). This means that $p$ is a polynomial of odd degree that has one real root at least. Because $p'(x)\neq0$ for every $x\in\R$, we take it that $p'(x)>0$ for every $x\in\R$ or $p'(x)<0$ for every $x\in\R$. If $\al_v>0$, then $p'(x)>0$ for every $x\in\R$ and $p$ is a strictly increasing function in $\R$, such that $\lim\limits_{x\ra+\infty}p(x)=+\infty$ and $\lim\limits_{x\ra-\infty}p(x)=-\infty$. Thus, there exists $x_0<0$ so that $p(x_0)<0$ and $y_0>0$ so that $p(y_0)>0$, which gives that $p$ has a real root in $(x_0,y_0)$, say $\rho$. Because $p$ is a strictly monotonous function, we take it that the root $\rho$ is the unique real root of $p$.

We can now compute some numbers $x_0,y_0$  with the above properties. By inequality $(\ast)$ in page 40 we take it that if $x\in\R$, so that $x>1$ and $x>\dfrac{|\al_0|+|\al_1|+\cdots+|\al_{v-1}|}{\al_v}$, then we get $\al_v+\dfrac{\al_v-1}{x}+\cdots+\dfrac{\al_1}{x^{v-1}}+\dfrac{\al_0}{x^v}>0$. We choose some $y_0>1$, so that $y_0>\dfrac{|\al_0|+|\al_1|+\cdots+|\al_v-1|}{\al_v}$, then by inequality $(\ast)$ in page 40 we get
\begin{align*}
&\al_v+\frac{\al_v-1}{y_0}+\cdots+\frac{\al_1}{y^{v-1}_0}+\frac{\al_0}{y^v_0}>0\Rightarrow\\
&y^v_0\bigg(\al_v+\frac{\al_v-1}{y_0}+\cdots+\frac{\al_1}{y^{v-1}_0}+\cdots+\frac{\al_0}{y^v_0}\bigg)>0\Leftrightarrow p(y_0)>0. \hspace*{2.5cm} {\mbox{(15)}}
\end{align*}
Now, let some $x<-1$. Then we have $(x\neq0)$
\[
\frac{\al_0}{x^v}\ge-\bigg|\frac{\al_0}{x^v}\bigg|,\ \ \frac{\al_1}{x^{v-1}}\ge-\bigg|\frac{\al_1}{x^{v-1}}\bigg|,\ld,\frac{\al_{v-1}}{x}\ge-\bigg|
\frac{\al_{v-1}}{x}\bigg|.
\]
Adding these inequalities we get:
\[
\frac{\al_0}{x^v}+\frac{\al_1}{x^{v-1}}+\cdots+\frac{\al_v-1}{x}\ge-
\bigg(\bigg|\frac{\al_0}{x^v}\bigg|+\bigg|\frac{\al_1}{x^{v-1}}\bigg|+\cdots+\bigg|
\frac{\al_{v-1}}{x}\bigg|\bigg)  \eqno{\mbox{(16)}}
\]
We get also
\begin{align*}
|x|>1&\Rightarrow|x|\ge|x|,|x^2|>|x|,\ld,|x^{v-1}|>|x|,|x^v|>|x|\\
&\Rightarrow\frac{1}{|x|}\le\frac{1}{|x|},\ld,\frac{1}{|x^{v-1}|}<
\frac{1}{|x|},\frac{1}{|x^v|}<\frac{1}{|x|} \\
&\Rightarrow\bigg|\frac{\al_{v-1}}{x}\bigg|\le\bigg|\frac{\al_{v-1}|}{|x|},\ld,
\bigg|\frac{\al_1}{x^{v-1}}\bigg|<\frac{|\al_1|}{|x|},\bigg|\frac{\al_0}{x^v}\bigg|<\bigg|\frac{\al_0}{x}\bigg|\\
&\Rightarrow-\bigg|\frac{\al_{v-1}}{x}\bigg|\ge-\bigg|\frac{\al_{v-1}}{x}\bigg|,\ld,-
\bigg|\frac{\al_1}{x^{v-1}}\bigg|>-\frac{|\al_1|}{|x|},-\bigg|\frac{\al_0}{x^v}\bigg|>-
\bigg|\frac{\al_0}{x}\bigg|.
\end{align*}
Adding by pairs the previous inequalities we get:
\[
-\bigg(\bigg|\frac{\al_0}{x^v}\bigg|+\bigg|\frac{\al_1}{x^{v-1}}\bigg|+\cdots+
\bigg|\frac{\al_{v-1}}{x}\bigg|\bigg)\ge-\frac{|\al_0|+|\al_1|+\cdots+|\al_{v-1}|}{|x|}.
\eqno{\mbox{(17)}}
\]
By (16) and (17) we get
\[
\al_v+\frac{\al_{v-1}}{x}+\cdots+\frac{\al_v}{x^{v-1}}+\frac{\al_0}{x^v}\ge\al_v+
\frac{|\al_0|+|\al_1|+\cdots+|\al_{v-1}|}{|x|}  \eqno{\mbox{(18)}}
\]
for every $x\in\R$, $x<-1$.

Now we get $x_0\in\R$, so that
\[
x_0<-1 \ \ \text{and} \ \ x_0<-\frac{|\al_0|+|\al_1|+\cdots+|\al_{v-1}|}{\al_v}.  \eqno{\mbox{(19)}}
\]
Then by (19) we get:
\begin{align*}
&-x_0>\frac{|\al_0|+|\al_1|+\cdots+|\al_{v-1}|}{\al_v}>0 \ \ \text{(because $\al_v>0$)} \\
&\Rightarrow|x_0|>\frac{|\al_0|+|\al_1|+\cdots+|\al_{v-1}|}{\al_v} \\
&\Rightarrow\al_v-\frac{|\al_0|+|\al_1|+\cdots+|\al_{v-1}|}{|x_0|}>0. \hspace*{5.5cm} {\mbox{(20)}}
\end{align*}
So, for $x<-1$, $x_0<-\dfrac{|\al_0|+|\al_1|+\cdots+|\al_{v-1}|}{\al_v}$ we get from (20), that
\[
\al_v+\frac{\al_{v-1}}{x_0}+\cdots+\frac{\al_1}{x^{v-1}_0}+\frac{\al_0}{x^v_0}>0\Rightarrow
\]
(because $x_0<0$ and $v$ is odd $x^v_0<0$)
\[
x^v_0\cdots\bigg(\al_v+\frac{\al_{v-1}}{x_0}+\cdots+\frac{\al_1}{x^{v-1}_0}+
\frac{\al_0}{x^v_0}\bigg)<0\Leftrightarrow p(x_0)<0.
\]
Thus, for some $x_0<-1$, $x_0<-\dfrac{|\al_0|+|\al_1|+\cdots+|\al_{v-1}|}{\al_v}$ we get $p(x_0)<0$ (21).

So with (15) and (12) we get $p(x_0)\cdot p(y_0)<0$ and by applying the bisection method we can compute the unique real root $\rho$ of $p$, so that $\rho\in(x_0,y_0)$.

Finally in the case of $\al_v<0$, we consider the polynomial $-p(x)$. Then
\[
(-p)'(x)=-p'(x)\neq0\ \  \text{for every} \ \  x\in\R
\]
and the coefficient of the monomial of greater degree of $-p$ is positive now. We apply the previous for $-p$ and we compute the unique real root $\rho$ of $-p$, that is the unique real root of $p$ also.

So in this remark we have covered the gap, we have left from basic Lemma 3.8 and Corollaries 3.9 and 3.10 and we have computed specific numbers $x_0,y_0$. For the sequel we also need some tools from real polynomials of two real variables.

Before this, let us give some specific examples for the roots of polynomials.

As it is well known, from elementary calculus, any polynomial of odd degree has a real root at least.

On the other hand, there are many polynomials of any even degree that do not have any real root. For example, let $p(x)$  be any polynomial that is non constant, let $k\in\N$, and $\thi>0$. Then polynomial $q(x)=p(x)^{2k}+\thi$ does not have any real root, as we can easily see, and has degree $2k\cdot v$, where $v=deg p(x)$.

Of course for every finite set of real numbers $A=\{\rho_1,\rho_2,\ld,\rho_v\}$, $v\in\N$, $\rho_i\neq\rho_j$, for $i,j\in\{1,2,\ld,v\}$, $i\neq j$, polynomial $p(x)=(x-\rho_1)(x-\rho_2)\ld(x-\rho_v)$ has roots the numbers $\rho_i$, $i=1,\ld,v$, and polynomial
\[
q(x)=((x-\rho_1)(x-\rho_2)\ld(x-\rho_v))^{2k}=p(x)^{2k}
\]
is a polynomial of even degree $deg q(x)=2kv$, with roots the numbers $\rho_i$, $i=1,2,\ld,v$, also. Now we consider polynomials of two real variables with real coefficients, that is we consider the set
\[
\R^2[x,y]=\{p(x,y):p(x,y)
\]
is a polynomial of two real variables $x$ and $y$ with coefficients in $\R\}$.\smallskip

Let $p(x,y)\in\R^2[x,y]$. We say that polynomial $p(x,y)$ is a pure polynomial, when $deg_x p(x,y)\ge1$ and $deg_yp(x,y)\ge1$ where $deg_xp(x,y)$, $deg_yp(x,y)$ are the greatest degree of its monomials with respect to $x$ (or $y$ respectively).

The set of roots of $p(x,y)$ is the set
\[
L_p(x,y)=\{(x,y)\in\R^2\mid p(x,y)=0\}.
\]
As in polynomials of one real variable, we can easily see that there are many pure polynomials $p(x,y)\in\R^2[x,y]$, that do not have any roots.

For example, let $p(x,y)$ be any pure polynomial. Then polynomial\linebreak $q(x,y)=p(x,y)^{2k}\!+\!\thi$, where $k\!\in\!\N$, $\thi\!>\!0$, is a pure polynomial that does not have any real root, as we easily see. Of course these polynomials are of even degree for $x$ and $y$.

On the other hand let $A=\{\al_1,\al_2,\ld,\al_v\}$, $B=\{\bi_1,\bi_2,\ld,\bi_m\}$, $A\cup B\subseteq\R$, $v,m\in\N$, $\al_i\neq\al_j$ for every $i,j\in\{1,\ld,v\}$ $i\neq j$ and $\bi_i\neq\bi_j$ for every $i,j\in\{1,2,\ld,m\}$, $i\neq j$. Let Also $k\in\N$. We consider the pure polynomial
\[
p(x,y)=((x-\al_1)(x-\al_2)\ld(x-\al_v))((y-\bi_1)(y-\bi_2)\ld(y-\bi_m))^{2k}. \]
Then $p$ is a pure polynomial of even degree with respect to $x$ and $y$ such that
\[
L_0=\{(\al_i,\bi_j),i\in\{1,\ld,v\},j\in\{1,2,\ld,m\}\}\subsetneqq L_p(x,y).
\]
We remark also that for ever $y\in\R$, the couple $(\al_i,y)$ is a root of $p(x,y)$. This fact differentiates pure polynomials $p(x,y)$ from polynomials of one variable.

That is, there exist uncountable pure polynomials, each one having uncountable set of real roots. Especially, this holds for pure polynomials of an odd degree with respect to $x$ and $y$. We have the following proposition.
\end{proof}
\noindent
{\bf Proposition 3.13.} {\em
Let $p(x,y)$ be a pure polynomial such that $deg_x(x,y)=v$ is odd or $deg p_y(x,y)$ is odd. Then for every $r\in\R$ the set $L_r=\{(x,y)\in\R^2:p(x,y)=r\}$ is uncountable}.
\begin{proof}
We suppose, without loss of generality, that number $v=deg p_x(x,y)$ is odd. Then, as we can see easily, we can write polynomial $p(x,y)$ as follows:
\[
p(x,y)=\al_v(y)x^v+\al_{v-1}(y)x^{v-1}+\cdots+\al_1(y)x+\al_0(y),
\]
where $\al_i(y)\in\R[y]$ for every $i=0,1,\ld,v$, and $\al_v(y)\neq0$, because $v=deg p_x(x,y)$. Because $\al_v(y)\neq0$, polynomial $\al_v(y)$ has a finite set of roots. Let $A_v$ be the set of roots of $\al_v(y)$, that is $A_v=\{y\in\R\mid\al_v(y)=0\}$. Let $y_0\in\R\sm A_v$. Then $\al_v(y_0)\neq0$. Also let $r\in\R$. We consider the polynomial
\[
p_r(x)=\al_v(y_0)x^v+\al_{v-1}(y_0)x^{v-1}+\cdots+\al_1(y_0)x+\al_0(y_0)-r.
\]
Then, $p_r(x)$ is a polynomial of odd degree $deg p_r(x)=v$, thus polynomial $p_r(x)$ has a real root, say $x_0$, at least, that is we have:
\[
\begin{array}{l}
  p_r(x_0)=0\Rightarrow\al_v(y_0)x^v_0+\al_{v-1}(y_0)x^{v-1}_0+\cdots+
\al_1(y_0)x_0+\al_0(y_0)=r\Leftrightarrow  \\ [1.5ex]
p(x_0,y_0)=r\Rightarrow(x_0,y_0)\in L_r.
\end{array}
\]
That is we proved that for every $y\in\R\sm A_v$, we have that there exists some $x\in\R$, such that $(x,y)\in L_r$. Of course, if $y_1,y_2\in\R\sm A_v$, $y_1\neq y_2$ and $(x_1,y_1),(x_2,y_2)\in L_r$, we have $(x_1,y_1)\neq(x_2,y_2)$ so the set $L_r$ is uncountable, and the proof of this proposition is complete. \qb
\end{proof}
\noindent
{\bf Corollary 3.14.} {\em
Let $p(x,y)$ be a pure polynomial such that $deg_xp(x,y)$ or $deg_yp(x,y)$ is odd. Then the set of real roots of $p(x,y)$ is uncountable}.
\begin{proof}
It is a simple application of the previous Proposition 3.13 for $r=0$.  \qb
\end{proof}

As we have noticed previously there are also such pure polynomials that the numbers $deg_xp(x,y)$ and $deg_yp(x,y)$ are even whose set of real roots is uncountable, as well as there being polynomials that do not have any roots.

Of course here we have the natural question: Are there pure polynomials $p(x,y)$ whose set of real roots $p(x,y)$ is non empty and finite? Of course, let us give a simple example:

We consider the polynomial: $p(x,y)=(x^2-1)^2+(y^2-4)^2$. It is easy to see that
\[
L_p(x,y)=\{(1,2),(1,-2),(-1,2),(-1,-2)\}.
\]
More generally, let $p_1(x)$ be a polynomial with real roots $\al_1,\al_2,\ld,\al_v$, and $p_2(y)$ be a polynomial with real roots $\bi_1,\bi_2,\ld,\bi_m$.

We consider the pure polynomial $p(x,y)=p_1(x)^{2k_1}+p_2(y)^{2k_2}$, where $k_1,k_2,\in\N$. Then, it is easy to see that:
\[
L_p(x,y)=\{(\al_i,\bi_j),i\in\{1,\ld,v\},j\in\{1,\ld,m\}\}.
\]
Of course by Corollary 3.14 only pure polynomials whose numbers $deg_xp(x,y)$, $deg_yp(x,y)$ are even can have finite set of roots, as in the previous examples. From the previous results we also  have a significant observation.

These polynomials have the number zero as a global minimum!

For the sequel, we have to concentrate our attention to pure polynomials $p(x,y)$ that have a finite set of roots. So, from the previous observation we are led to ask whether the reverse result holds. That is, does any pure polynomial that has a global minimum have a finite set of real roots? The answer is no, and we can give a simple example. We consider the pure polynomial $p(x,y)=((x-1)(y-2))^2-7$. It is easy to check that polynomial $p(x,y)$ has the number $-7$ as a global minimum. For this polynomial we get:
\[
p(x,3)=(x-1)^2-7 \ \ \text{for every} \ \ x\in\R
\]
so $p(1,3)=-7<0$ and $p(4,1)=2>0$, thus there exists $x_0\in(1,4)$ so that $p(x_0,3)=0$. Similarly take any real number $y_0\in(3,4)$. That is $3<y_0<4\Rightarrow1<y_0-2<2\Rightarrow1<(y_0-2)^2<4$ (1).

We get:
\[
\begin{array}{l}
  p(1,y_0)=-7<0 \\ [1.5ex]
  p(8,y_0)=7^2(y_0-2)^2>7^2>0
\end{array}
\]
by (1), so there exists $x_1\in(1,8)$ such that $p(x_1,y_0)=0$.

Thus, for every $y\in(3,4)$, there exists $x\in\R$, so that $p(x,y)=0$, and of course if $y_1,y_2\in(3,4)$, $x_1,x_2\in\R$, $p(x_1,y_1)=p(x_2,y_2)=0$ and $y_1\neq y_2$ we have $(x_1,y_1)\neq(x_2,y_2)$, so the set of real roots of $p$ is uncountable even if polynomial $p(x,y)$ has a global minimum. However, the property of a pure polynomial to have a global minimum (or maximum also) is a crucial property that have all pure polynomials that have a finite number of roots, as we will prove now with the following proposition.\medskip\\
\noindent
{\bf Proposition 3.15.} {\em
(topological lemma). Let $p(x,y)$ be a pure polynomial. We suppose that there exist two couples $(x_1,y_1),(x_2,y_2)\in\R^2$ such that\linebreak $p(x_1,y_1)\cdot p(x_2,y_2)<0$. Then, set $L_p(x,y)$ is uncountable}.
\begin{proof}
We set $A=(x_1,y_1)$, $B=(x_2,y_2)$. We get $A\neq B$, or otherwise we have $A=B$ and $p(x_1,y_1)\cdot p(x_2,y_2)=p(A)\cdot p(B)=p(A)^2\ge0$, which is false. So we get $A\neq B$. We consider the midperpendicular $\el$ of segment $[A,B]$. For every point $\Ga\in\el$, we consider the union of two segments $[A,\Ga]\cup[\Ga,B]$. We write $A\Ga B=[A,\Ga]\cup[\Ga,B]$ for simplicity. Of
course $A\Ga B\subseteq\R^2$. We consider the restriction $p|_{A\Ga B}$ for simplicity, and we write $p=p|_{A\Ga B}$ also for simplicity.

Of course the set $A\Ga B$ is a compact and connected subset of $\R^2$. So the set $p(A\Ga B)$ os a closed interval of $\R$.

We suppose that $p(A)<0$ and $p(B)>0$, without loss of generality. So $p(A),p(B)\in p(A\Ga B)$ and gives that $0\in p(A\Ga B)$, that is there exists some point $\De\in A\Ga B$ so that $p(\De)=0$. Of course $\De\neq A$ and $\De\neq B$. So, for every $\Ga\in\el$ and every curve $A\Ga B$, there exists some $\De\in A\Ga B$, $\De\neq A$, $\De\neq B$, such that $p(\De)=0$.

Because the set $\ca=\{A\Ga B,\Ga\in\el\}$ is an uncountable supset of $P(\R^2)$ (the powerset of $\R^2$), and for every $\Ga_1,\Ga_2\in\el,\Ga_1\neq\Ga_2$, we have that $A\Ga_1B\cap A\Ga_2B=\{A,B\}$ this means that the set
\[
\cb=\{\De\in A\Ga B\mid\Ga\in\el \ \ \text{and} \ \ p(\De)=0\}
\]
is uncountable, that gives that the set $Lp(x,y)$ of roots of $p(x,y)$ is uncountable and the proof of this proposition is complete. \qb
\end{proof}
\noindent
{\bf Corollary 3.16.} {\em
Let $p(x,y)$ be a pure polynomial that has a finite set of roots, non empty. Then, number 0 is the global minimum or maximum of $p(x,y)$, or in other words polynomial $p(x,y)$ has a global maximum or minimum, and when this holds, then this global maximum or minimum is number 0}.
\begin{proof}
There exists no two points $(x_1,y_1),(x_2,y_2)\in\R^2$ so that:
\[
p(x_1,y_1)\cdot p(x_2,y_2)<0.
\]
Or else, if there exist two points $(x_1,y_1),(x_2,y_2)\!\in\!\R^2$, so that $p(x_1,y_1)\!\cdot\! p(x_2,y_2)<0$, then set $L_p(x,y)$ is uncountable (by the previous Proposition 3.13), which is false by our supposition.

This means that we have:

i) $p(x,y)\ge0$ for every $(x,y)\in\R^2$, or

ii) $p(x,y)\le0$ for every $(x,y)\in\R^2$.

We suppose that i) holds. Because set $Lp(x,y)$ is non-empty, this means that there exists $(x_0,y_0)\in\R^2$ so that $p(x_0,y_0)=0$, so we get $p(x,y)\ge p(x_0,y_0)$ for every $(x,y)\in\R^2$. So, polynomial $p(x,y)$ has in point $(x_0,y_0)$ its global minimum the number 0, because $p(x_0,y_0)=0$. If ii) holds, then we take with a similar way that $p$ has  global maximum the number 0 in a point, and the proof of corollary is complete. \qb
\end{proof}

The above corollary is a basic result that we use in the second stage of our method.

Finally, we refer here the most advanced result, that we use in our method.

This result is called many times as Fermat's Theorem in calculus of several variables.\medskip \\
\noindent
{\bf Theorem 3.17.} {\em
Let $U\subseteq\R^2$, $U$ open and $f:U\ra\R$ be a differentiable function in $x_0\in U$, where $x_0$ is a point of local maximum or local minimum of $f$. Then the following holds: $\bigtriangledown f(x_0)=0$, that is $x_0$ is a crucial point of $f$, where $\nabla f(x_0)$ is the gradient of $f$ in $x_0$}.
\section{Appendix}
\noindent

Fundamental Theorem of Algebra is a powerful and basic result in the theory of polynomials, especially in polynomial equations.

Gauss gave the first complete proof of this result in his Ph.D. There are many proofs for this important Theorem but none of them is trivial in order to be presented in books of secondary school.

Its simplest proof comes from complex analysis and uses an advanced Theorem of complex analysis, Liouville's Theorem. Here we give a proof that uses the most elementary tools that an undergraduate student learns.

We think that it is difficult for an undergraduate student to find this proof in books, so we try to present it with details for educational reasons.

For this reason we give firstly some elementary lemmas.\medskip \\
\noindent
{\bf Lemma 4.1} {\em
Let $p(z)\in\C[z]$ be a complex polynomial, and $z_0\in\C$. We consider polynomial $Q(z)=p(r+z_0)$, $z\in\C$. If $p(z)\equiv0$, then of course $Q(z)\equiv0$. If $p(z)\not\equiv0$, then $deg Q(z)=deg p(z)$}.
\begin{proof}
If $deg p(z)=0$, then the result is obvious. Let $deg p(z)=n\in\N$, $n\ge1$. We suppose that $n=1$, so we get $p(z)=az+b$, where $a,b\in\C$, $a\neq0$. We get: $Q(z)=p(z+z_0)=a(r+z_0)+b=az+(az_0+b)$, and $deg Q(z)=1$, because $a\neq0$. So, the result holds for $n=1$. We prove the result inductively. We suppose $z_0\neq0$.

For $n=1$, the result holds.

We suppose that result holds for $k\in\N$, $k\ge1$ and for every $j\in\N$, $1\le j\le k$. We prove that result holds for $k+1$.

We suppose that
\[
p(z)=a_0+a_1z+\cdots+a_kz+a_{k+1}z^{k+1} \ \ \text{and} \ \ a_{k+1}\neq0, \ \ \text{so} \ \ deg p(z)=k+1.
\]
We distinguish two cases:

(i) $q(z)=a_0+a_1z+\cdots+a_kz^k\not\equiv0$. \\
Then we have: $p(z)=q(z)+a_{k+1}z^{k+1}$.

We have
\[
Q(z)=p(z+z_0)=q(z+z_0)+a_{k+1}(z+z_0)^{k+1}.  \eqno{\mbox{(1)}}
\]
We set $r(z)=q(z+z_0)$, $z\in\C$. Because $q(z)\not\equiv0$, by induction step we have $deg r(z)=deg q(z)\le k$ (2).

We have by Newton's binomial
\[
a_{k+1}(z+z_0)^{k+1}=a_{k+1}\sum^{k+1}_{j=0}z^{k+1-j}z^j_0
=\sum^{k+1}_{j=0}a_{n+1}z^j_0z^{k+1-j}.  \eqno{\mbox{(3)}}
\]
Because $z_0\neq0$ (by our supposition) and $a_{k+1}\neq0$ we have $deg a_{k+1}(z+z_0)^{k+1}=k+1$ (4), by equality (3).

By (1), (2) and (4), we get $deg Q(z)=k+1$ and the result holds.

(ii) $q(z)=a_0+a_1z+\cdots+a_kz^k\equiv0$. The proof is similar to case (i), so the result holds by induction.

Of course if $z_0=0$ the result is obvious, because $Q(z)=p(z)$, so $deg Q(z)=deg p(z)$.
\end{proof}
\noindent
{\bf Lemma 4.2.} {\em
We consider polynomial $p(z)\in\C[z]$. Of course we have $|p(z)|\ge0$ for every $z\in\C$. So the set
\[
A=\{x\in\R\mid\exists\;z\in\C:x=|p(z)|\}
\]
is low bounded by 0.

We set $m=\inf(A)$. Then there exist $R>0$ so that:
\[
m=\inf(\{x\in\R\mid\exists z\in\overline{D(0,R)}:x=|p(z)|\})
\]
where $\overline{D(0,R)}=\{z\in\C:|z|\le R\}$}.
\begin{proof}
We set $B_R=\{x\in\R\mid\exists z\in D(0,R):x=|p(z)|\}$ for some $R>0$.

The result is obvious when $p(z)\equiv0$, so we suppose that $p(z)\not\equiv0$. It is obvious that $B_R\subseteq A$ by definitions of sets $A$ and $B_R$ for $R>0$. Let $x\in B_R$ for some $R>0$. Then $x\in A$. So: $m\le x$, because $m$ is a lower bound of $A$. So we have: $m\le x$ for every $x\in B_R$. This means that $m$ is a lower bound of $B_R$, so: $m\le m^\ast_R$ (1), where $m^\ast_R=\inf(B_R)$. That is we get: $m\le m^\ast_R$ for every $R>0$.

We suppose that:\\
$p(z)=a_0+a_1z+\cdots+a_nz^n$, where $n\in\N\cup\{0\}$, $a_n\neq0$. When $n=0$, we get of course $m^\ast_R=m=|p(z)|$ for every $z\in\C$ and every $R>0$, and the result is obvious of course. So we suppose that $n\ge1$.

Then for every $z\in\C\sm\{0\}$ we get
\[
p(z)=z^n\cdot\bigg(\frac{a_0}{z^n}+\frac{a_1}{z^{n-1}}+\cdots+\frac{a_{n-1}}{z}+a_n\bigg).
\]
By calculus of the elementary limits in complex analysis we have:\\
$\lim\limits_{z\ra\infty}\dfrac{a_0}{z^n}=\lim\limits_{z\ra\infty}\dfrac{a_1}{z^{n-1}}=
\cdots=\lim\limits_{z\ra\infty}\dfrac{a_{n-1}}{z}=0$ and $\lim\limits_{z\ra\infty}z^n=\infty$, so we have: $(a_n\neq0)$ $\lim\limits_{z\ra\infty}p(z)=\infty$. By definition of $\lim\limits_{z\ra\infty}p(z)$, this means that: for $m+1$, there exists $R_0>0$, so that: $|p(z)|>m+1$ for every $z\in\C$, $|z|>R_0^{(\ast)}$.

From (1) we have of course $m\le m^\ast_{R_0}$ (2). Take $w\in\C$: $|w|>R_0$. Then, by the above we have $|p(w)|>m+1$ (3).

Now there exists $z_1\in\C$ so that $|p(z_1)|<m+1$ (4), or otherwise we have\linebreak $|p(z)|\ge m+1$ for every $z\in\C$, so $m+1$ is a lower bound of $A$, that is\linebreak $m=\inf(A)\ge m+1$, which is false. Of course $z_1\in\overline{D(0,R_0)}$ by implication $(\ast)$, or else $|z_1|>R_0$ that means $|p(z_1)|>m+1$ (5) that is false by the above inequalities (4) and (5). So we have $m^\ast_{R_0}\le|p(z_1)|<m+1<|p(w)|\Rightarrow m^\ast_{R_0}\le|p(w)|$.

So we get: $m^\ast_{R_0}\le|p(z)|$ for every $z\in\C:|z|>R_0$. Of course we have also\linebreak $m^\ast_{R_0}\le|p(z)|$ for every $z\in\overline{D(0,R_0)}$ by definition of $m^\ast_{R_0}$. So we get $m^\ast_{R_0}\le|p(z)|$ for every $z\in\C$, that means that $m^\ast_{R_0}$ is a lower bound of $A$, that is $m^\ast_{R_0}\le m$ (6). From (2) and (6) we get $m=m^\ast_{R_0}$, that is Lemma 4.2 has been proven. \qb
\end{proof}
\noindent
{\bf Remark 4.3.} {\em
{\bf De Moivre Theorem:} We remind here the following result.\\
Let $n\in\N$, $n\ge2$. Then every non-zero complex number has exactly $n$ roots, that is if $w\in\C$, $w\neq0$, then equation $z^n=w$ has exactly $n$ solutions: This result is proven easily by elementary properties of complex numbers and it is well known as De Moivre's Theorem, using properties of functions sine and cosine. We also need a topological Theorem.} \medskip \\
\noindent
{\bf Theorem 4.4} {\em
Let $K\subseteq\C$ be compact and $f:K\ra\R$ be continuous. Then $f$ attains its supremum and its infimum and both are finite. For this theorem see \cite{5}. After the above, we are now ready to give the proof of fundamental Theorem of Algebra}. \medskip\\
\noindent
{\bf Fundamental Theorem of Algebra 4.5}
\begin{proof}
We consider polynomial
\[
p(z)=a_0+a_1z+\cdots+a_nz^n, \ \ a_i\in\C \ \ \text{for every} \ \ i=0,1,\ld,n, \ \ n\in\N, \ \ a_n\neq0, \ \ n\ge1.
\]
We prove that $p(z)$ has a root that is there exists $z_0\in\C$, so that $p(z_0)=0$. First of all we examine the case of $a_n=1$.

Of course we have $|p(z)|\ge0$ for every $z\in\C$. We set
\[
A=\{x\in\R\mid\exists z\in\C:x=|p(z)|\}.
\]
Set $A$ is low bounded by 0. We set $m=\inf(A)$. Of course $m\ge0$. For every $R>0$ we set:
\[
B_R=\{x\in\R\mid\exists z\in\overline{D(0,R)}:x=|p(z)|\}, \ \ \text{and}
\]
\[
m^\ast_R=\inf(B_R),\overline{D(0,R)}=\{z\in\C:|z|\le R\}.
\]
Applying Lemma 4.2 we take that there exists $R_0>0$ so that: $m=m^\ast_{R_0}$ (1).

By Theorem 4.3, page 233 [5], Ball $\overline{D(0,R_o)}=\{z\in\C:|z|\le R_0\}$ is a compact set as a set closed and bounded. Polynomial $p$ is a continuous function in $\C$. This is a well known result in elementary Complex analysis.Usual norm $|\;|:\C\ra\R$ is a continuous function also in $\C$, by
elementary complex analysis. So, the composition function $F:\C\ra\R$, $F=|\cdot|\circ p$, where $p:\C\ra\C$, $|\cdot|:\C\ra\R$ with formula\linebreak $F(z)=(|\cdot|\circ p)(z)=|p(z)|$ for every $z\in\C$ is a continuous functions as the composition of continuous functions $|\cdot|$ and $p$. Applying now Theorem 4.4 for $K=\overline{D(0,R_0)}$ and $f=F$ we take it that function $F$ attains its infimum is some point $z_0\in\overline{D(0,R_0)}$. This means that $|p(z_0)|=m^\ast_{R_0}$ (2). By (1) and (2) we have $m=|p(z_0)|$ (3). We argue that $m=0$. To take a contradiction we suppose that $m>0$. Because $|p(z_0)|=m>0$, we see that $p(z_0)\neq0$.

We consider polynomial $Q(z)=\dfrac{p(z+z_0)}{p(z_0)}$, that is defined well because $p(z_0)\neq0$.

Applying Lemma 4.1 we see that $deg p(z+z_0)=deg p(z)=n$, and by definition of $Q(z)$, we get: $deg Q(z)=n$. We have $Q(0)=\dfrac{p(0+z_0)}{p(z_0)}=1$, so polynomial $Q(z)$ has constant term equal to 1.

Let
\[
Q(z)=1+c_kz^k+\cdots+c_nz^n, \ \ c_n\neq0, \ \ \text{for every} \ \ z\in\C, \ \ \text{where} \ \ k\in\N, \ \ 1\le k\le n
\]
and $k$ be the smallest natural number such that $c_k\neq0$, (maybe $k=n$ of course).

So, we get: $-|c_k|/c_k\neq0$. From Remark 4.3 there exists $j\in\C$, so that\linebreak$j^k=-|c_k|/c_k$ (4). (Of course there are $k$ different complex numbers such that (4) holds). By (4) we take $|j^k|=|-|c_k|/c_k|=1\Rightarrow|j|=1$ (5).

By choice of $j$ we have: for $r\in\C$ $|1+c_kr^kj^k|\overset{(4)}{=}|1+c_kr^k\cdot(-|c_k|/c_k)=1-|c_k|r^k$ (6).

By definition of $Q(z)$ we compute for $z=rj$ for $r\in\C$:
\begin{align*}
Q(z)&=Q(rj)=1+c_k(rj)^k+\cdots+c_n(rj)^n\\
&=1+c_kr^kj^k+c_{k+1}r^{k+1}j^{k+1}+\cdots+
c_nr^nj^n. \hspace*{4.4cm} (7)
\end{align*}
By (7) and triangle inequality we get:
\[
|Q(rj)|\le|1+c_kr^kj^k|+|c_{k+1}r^{k+1}j^{k+1}|+\cdots+|c_nr^nj^n|.  \eqno{\mbox{(8)}}
\]
Applying (6) we get by (8)
\begin{align*}
|Q(rj)|&\le1-|c_k||r^k|+|c_{k+1}||r|^{k+1}+\cdots+|c_n||r^n| \\
&=1-|r^k|(|c_k|-|c_{k+1}||r|-\cdots-|c_n||r|^{n-k}), \ \ \text{for every} \ \ r\in\C\hspace*{1.3cm} (10)
\end{align*}
By definition of $m$ we get:
\[
m\le|p(z+z_0)| \ \ \text{for every} \ \ z\in\C. \eqno{\mbox{(11)}}
\]
By (3) and (11) we get:
\begin{align*}
|p(z_0)|\le|p(z+z_0)| \ \ &\text{for every} \ \ z\in\C \\
\Rightarrow \bigg|\frac{p(z+z_0)}{p(z_0)}\bigg|\ge1 \ \ &\text{for every} \ \ z\in\C\\
\Rightarrow |Q(z)|\ge1 \ \ &\text{for every} \ \ z\in\C \ \ \text{(by definition of $Q$)}  \hspace*{3cm} (12)
\end{align*}
Now, we distinguish two cases:

(i) $k=n$. Then, from (10) we get:
\[
|Q(rj)|\le1-|r|^k|c_k|,  \eqno{\mbox{(11)}}
\]
So, for every $r\neq0$, we get by (12)
\[
|Q(rj)|\ge 1  \ \ \text{and} \eqno{\mbox{(13)}}
\]
\[
|Q(rj)|\le1-|r^k||c_k|<1  \eqno{\mbox{(14)}}
\]
and we take a contradiction from (13) and (14).

(ii) $k<n$.

By properties of complex limits we get:
\[
\lim_{r\ra0}(|c_k|-|c_{k+1}||r|-\cdots-|c_n|r|^{n-k})=|c_k|>0.
\]
This limit shows us that: there exists some small $r_0$ so that
\[
|c_k|-|c_{k+1}||r_0|-\cdots-|c_n||r_0|^{n-k}>0.  \eqno{\mbox{(15)}}
\]
We set $\thi:=|c_k|-|c_{k+1}||r_0|-\cdots-|c_n||r_0|^{n-k}$. So we have $\thi_0>0$. From (10) and (15) we get:
\[
|Q(r_0j)|\le1-|r_0|^k\thi_0<1. \eqno{\mbox{(16)}}
\]
From (12) we get: $|Q(r_0j)|\ge1$ (17). By (16) and (17) we get a contradiction. So, our supposition that $m>0$ is false. So we have $m=0$ and from (3) we get $0=p(z_0)$, that is polynomial $p$ has, as a root number $z_0$. If $a_n\neq1$ we write: \\ $\dfrac{1}{a_n}p(z)=\dfrac{a_0}{a_n}+\dfrac{a_1}{a_n}z+\cdots+z^n$, and applying the previous result we take it that there exists some $w\in\C$ such that $\dfrac{1}{a_n}p(w)=0\Leftrightarrow p(w)=0$, so polynomial $p$ has a root again. The proof of fundamental Theorem has completed now. \qb
\end{proof}
\noindent
{\bf Remark 4.6.} {\em
Inside our work we have used the well known binomial equation\linebreak $x^n=a$, where $a>0$. We remind how we solve this equation here, for $n\ge2$, $n\in\N$. We will distinguish two cases:\medskip

(i) $a>1$. We consider function $f:[1,a]\ra\R$, with the formula $f(x)=x^n-a$ for every $x\in[1,a]$. We get $f(1)=1^n-a<1$, from our supposition and\\
$f(a)=a^n-a=a(a^{n-1}-1)>0$. So we have $f(1)\cdot f(a)<0$ and because $f$ is continuous we understand from Bolzano Theorem that there exists  $x_0\in(1,a)$ so that: $f(x_0)=0\Leftrightarrow x^n_0-a=0\Leftrightarrow x_0=\sqrt[n]{a}$. Because $f$ is strictly increasing in $[1,a]$, (because $f'(x)=nx^{n-1}>0$ for every $x\in[1,a]$) equation $f(x)=0$ has unique root in $[1,a]$, that is number $\sqrt[n]{a}$. Applying bisection method we approximate number $\sqrt[n]{a}$, or in other words we solve the equation $x^n=a$.\medskip

(ii) $a\in(0,1)$. Then we apply the above procedure similarly to the function\linebreak $g:[0,1]\ra\R$ with the formula $g(x)=x^n-a$, for every $x\in[0,1]$}.\vspace*{0.5cm} \\
\noindent
{\bf Acknowledgements:}
Many thanks to Vasilli Karali for his contribution in the presentation of  this paper.
\vspace*{1.5cm}
\noindent
Nikos Tsirivas,\\
Department of Mathematics, University of Patras and\\
Department of Marine Engineering, University of West Attica.

\end{document}